\documentclass[11pt, notitlepage]{article}
\usepackage{amssymb,amsmath,comment}
\usepackage{amsthm}
\catcode`\@=11 \@addtoreset{equation}{section}
\def\thesection{\arabic{section}}

\def\theequation{\thesection.\arabic{equation}}
\catcode`\@=12
\usepackage{colortbl}
\usepackage{hyperref}
\usepackage[mathscr]{eucal}
\usepackage{epsf}
\usepackage{esint}
\usepackage{a4wide}
\def\R{\mathbb{R}}

\newcommand{\noi} {\noindent}

\usepackage[all]{xy}
\catcode`\@=11
\def\theequation{\@arabic{\c@section}.\@arabic{\c@equation}}
\catcode`\@=12

\newtheorem{Theorem}{Theorem}[section]
\newtheorem{Lemma}[Theorem]{Lemma}

\newtheorem{Remark}[Theorem]{Remark}
\newtheorem{Definition}[Theorem]{Definition}

\begin{document}

{\vspace{0.01in}}

\title{Regularity and existence for semilinear mixed local-nonlocal equations with variable singularities and measure data}

\author{Sanjit Biswas\footnote{Department of Mathematics and Statistics, Indian Institute of Technology Kanpur, Kanpur-208016, Uttar Pradesh, India, Email: sanjitbiswas410@gmail.com } \,\,and Prashanta Garain\footnote{(Corresponsding author) Department of Mathematical Sciences, Indian Institute of Science Education and Research, Berhampur, 760010, Odisha, India, Email: pgarain92@gmail.com}}

\maketitle

\begin{abstract}\noindent
This article proves the existence and regularity of weak solutions for a class of mixed local-nonlocal problems with singular nonlinearities. We examine both the purely singular problem and perturbed singular problems. A central contribution of this work is the inclusion of a variable singular exponent in the context of measure-valued data. Another notable feature is that the source terms in both the purely singular and perturbed components can simultaneously take the form of measures. To the best of our knowledge, this phenomenon is new, even in the case of a constant singular exponent.
\end{abstract}

\maketitle

\noi {Keywords: Mixed local-nonlocal singular equation, existence, regularity, variable exponent, measure data.}

\noi{\textit{2020 Mathematics Subject Classification: 35M10, 35M12, 35J75, 35R06, 35R11, 35B65}

\bigskip

\tableofcontents
\section{Introduction}
In this article, we explore the existence and regularity properties of weak solutions for the following mixed local-nonlocal measure data problem with variable singular exponent
\begin{align}\label{ME}
     \begin{cases}
        &\mathcal{M}u:=-\mathcal{A}u+\mathcal{B}u=\frac{\nu}{u^{\delta(x)}}+\mu \text{ in } \Omega,\\
        &u=0 \text{ in } \mathbb{R}^N\setminus \Omega \text{ and } u>0 \text{ in }\Omega,
    \end{cases}
\end{align}
where $\Omega\subset\mathbb{R}^N,\,N\geq 2$ is a bounded domain with Lipschitz boundary. Here $\mathcal{A}u=\text{div}(A(x)\nabla u)$, where $A:\Omega\to\mathbb{R}^{N^2}$ is a bounded elliptic matrix satisfying
\begin{equation}\label{lkernel}
\alpha|\xi|^2\leq A(x)\xi\cdot\xi,\quad |A(x)|\leq\beta,
\end{equation}
for every $\xi\in\mathbb{R}^N$ and for almost every $x\in\Omega$, for some constants $0<\alpha\leq\beta$. Further, $\mathcal{B}$ represents the nonlocal Laplace operator defined as follows:
$$
\mathcal{B}u=\text{P.V.}\int_{\mathbb{R}^N}(u(x)-u(y))K(x,y)\,dy,
$$
where P.V. denotes the principal value and $K$ is a symmetric kernel in $x$ and $y$ that satisfies
\begin{equation}\label{nkernel}
\frac{\Lambda^{-1}}{|x-y|^{N+2s}}\leq K(x,y)\leq\frac{\Lambda}{|x-y|^{N+2s}}
\end{equation}
for some constant $\Lambda\geq 1$ and $0<s<1$. Notably, if we take $A(x)=I$ and $K(x,y)=|x-y|^{-N-2s}$, the operators $\mathcal{A}$ and $\mathcal{B}$ reduces to the usual Laplace operator $-\Delta$ and the fractional Laplace operator $(-\Delta)^s$ respectively and consequently, the operator $\mathcal{M}$ simplifies to the mixed local-nonlocal Laplace operator $-\Delta+(-\Delta)^s$. Hence, equation \eqref{ME} serves as an extension of the following mixed local-nonlocal singular problem:
\begin{align}\label{ME1}
     \begin{cases}
        &-\Delta u+(-\Delta)^s u=\frac{\nu}{u^\delta}+\mu \text{ in } \Omega,\\
        &u=0 \text{ in } \mathbb{R}^N\setminus \Omega \text{ and } u>0 \text{ in }\Omega.
    \end{cases}
\end{align}

We assume that $\mu$ and $\nu$ are non-negative bounded Radon measures on $\Omega$, with $\nu$ not being identically zero. Further, we assume that $\delta:\overline{\Omega}\to(0,\infty)$ is a continuous function. The positivity of $\delta$ leads to a blow-up of the nonlinearity in \eqref{ME} near the origin, a phenomenon referred to as singularity. Consequently, equation \eqref{ME} encompasses a broad spectrum of mixed singular problems, including both constant and variable exponent singular nonlinearities with measure data.

In the purely local case, the singular Laplace equation
\begin{align}\label{lap}
     \begin{cases}
        &-\Delta u=\frac{f}{u^{\delta(x)}}\text{ in } \Omega,\\
        &u=0 \text{ on } \partial\Omega \text{ and } u>0 \text{ in }\Omega
    \end{cases}
\end{align}
is widely studied for both the constant and variable exponent $\delta$. {When $\delta$ is a positive constant, existence of a unique classical solution is obtained in \cite{CRT} under the assumption that $\partial\Omega$ is of class $C^3$ and $f\in C^1(\overline{\Omega})\setminus\{0\}$ is non-negative. Indeed, authors in \cite{CRT} obtained existence results for more general singularity and more general operator.} For constant $\delta>0$, existence of weak solutions is also obtained in \cite{BocOrs}. Moreover, when $f$ is a non-negative bounded Radon measure on $\Omega$, existence results can be found in \cite{OPdie}. When $\delta$ is a variable, for some positive $f\in L^m(\Omega)$ with $m\geq 1$, existence results are established in \cite{CMvar} in the semilinear case and for the associated quasilinear equations, we refer the reader to \cite{Alvesjde, BGM, Garainmm} and the references therein.

Further, the purturbed singular Laplace equation
\begin{align}\label{purlap}
     \begin{cases}
        &-\Delta u=\frac{f}{u^{\delta(x)}}+g\text{ in } \Omega,\\
        &u=0 \text{ in } \mathbb{R}^N\setminus \Omega \text{ and } u>0 \text{ in }\Omega
    \end{cases}
\end{align}
is also studied. When $\delta$ is a positive constant and both $f,g$ are some non-negative integrable functions, multiplicity of weak solutions is obtained in \cite{Arcoyadie, Arcoyana} and the references therein. In this concern, when $f$ is some non-negative integrable function and $g$ is some non-negative bounded Radon measure, existence results are established in \cite{POesaim} and the references therein. {Further, measure data problems for Laplace equation is studied in \cite{ref2, ref3}.}

In the purely nonlocal case, the singular fractional Laplace equation
\begin{align}\label{flap}
     \begin{cases}
        &(-\Delta)^s u=\frac{f}{u^{\delta(x)}}\text{ in } \Omega,\\
        &u=0 \text{ in } \mathbb{R}^N\setminus \Omega \text{ and } u>0 \text{ in }\Omega
    \end{cases}
\end{align}
is also widely studied. Indeed, when $\delta$ is a positive constant, existence of a unique classical solution is established in \cite{Fang}, provided $0<\delta<1$ and {$f=1$ in a bounded smooth domain $\Omega$}. The case of any $\delta>0$ is tackled in \cite{Sciunzi} to obtain weak solutions, where the nonlinear version of \eqref{flap} is also studied. When $f$ is a non-negative bounded Radon measure, existence results can be found in \cite{Giri} and the references therein. When $\delta$ is a variable, existence and regularity results are obtained in \cite{GMcpaa} for the semilinear and quasilinear cases, provided $f\in L^m(\Omega)\setminus\{0\}$ is non-negative for some $m\geq 1$.

Further, existence results for the purturbed singular fractional Laplace equation
\begin{align}\label{fpurlap}
     \begin{cases}
        &(-\Delta)^s u=\frac{f}{u^{\delta(x)}}+g\text{ in } \Omega,\\
        &u=0 \text{ in } \mathbb{R}^N\setminus \Omega \text{ and } u>0 \text{ in }\Omega
    \end{cases}
\end{align}
is proved in \cite{Giri}, provided $\delta$ is a positive constant, $f$ is some non-negative integrable function and $g$ is some non-negative bounded Radon measure on $\Omega$. For variable $\delta$, we refer to \cite{PKvar} in the semilinear case and in the quasilinear case, see \cite{GMcpaa} and the references therein.

In the recent years, mixed local-nonlocal problems has drawn a great attention due to its wide range of applications in biology, stocastic processes, image processing etc., see \cite{Valdinoci} and the references therein. The non-singular mixed local-nonlocal problem
\begin{align}\label{mpurlap}
     \begin{cases}
        &-\Delta u+(-\Delta)^s u=f\text{ in } \Omega,\\
        &u=0 \text{ in } \mathbb{R}^N\setminus \Omega \text{ and } u>0 \text{ in }\Omega
    \end{cases}
\end{align}
is studied concerning existence, regularity and several other qualitative properties in \cite{Valdinocicpde, Biswas, ref4} in the presence of integrable functions $f$. Further, the nonlinear analogue of equation \eqref{mpurlap} is also studied in \cite{Mingione, GKK, GKtams, GL, ref1} and the references therein. Equation \eqref{mpurlap} is recently studied in the presence of measure $f$ in \cite{Byun, Iwona}.

When $\delta$ is a positive constant, the purely singular mixed local-nonlocal problem
\begin{align}\label{mlap1}
     \begin{cases}
        &-\Delta u+(-\Delta)^s u=\frac{f}{u^{\delta(x)}}\text{ in } \Omega,\\
        &u=0 \text{ in } \mathbb{R}^N\setminus \Omega \text{ and } u>0 \text{ in }\Omega
    \end{cases}
\end{align}
is studied in \cite{Arora, Gjgea, Hichem} concerning existence, regularity and the quasilinear version of \eqref{mlap1} is studied in \cite{Guna}. For variable exponent $\delta$, existence and regularity is obtained in \cite{Biroud, GKK} and the references therein. 

The purturbed mixed {local-nonlocal} problem
\begin{align}\label{mlap2}
     \begin{cases}
        &-\Delta u+(-\Delta)^s u=\frac{f}{u^{\delta(x)}}+g\text{ in } \Omega,\\
        &u=0 \text{ in } \mathbb{R}^N\setminus \Omega \text{ and } u>0 \text{ in }\Omega
    \end{cases}
\end{align}
is also studied in \cite{Gjgea} and the references therein, provided $0<\delta<1$ is a constant. Recently the case $\delta\geq 1$ is also settled in \cite{BalDas}, where the authors studied the quasilinear analogue of \eqref{mlap2} as well.

Very recently, equation \eqref{mlap2} is studied in \cite{Ghosh}, when $\delta$ is a positive constant, $f$ is a positive integrable function and $g$ is a non-negative bounded Radon measure in $\Omega$. To the best of our knowledge, mixed local-nonlocal problems are not understood in the presence of a measure data with variable singular exponent. Our main purpose in this article is to fill this gap. We would like to emphasize that some of our results are valid, even when both $f$ and $g$ are measures. As far as we are aware, such phenomenon is new even in the constant singular exponent case. Further, we remark that the operator $\mathcal{M}$ is more general than the mixed operator $-\Delta+(-\Delta)^s$ and therefore, our main results are valid for the more general mixed equation \eqref{ME}.

To demonstrate our main results, we adopt the approximation approach outlined in \cite{BocOrs, POesaim, OPdie}. Specifically, we establish the existence of solutions to the approximated problem using fixed point arguments. Subsequently, we take the limit, which necessitates several a priori estimates. These estimates are derived by selecting appropriate test functions for the approximated problem.

The structure of this article is as follows: Section 2 presents the functional framework and states the main results. {Sections 3 and 4 are dedicated to the proofs of the existence and regularity results respectively. Finally, in the appendix section 5, approximate problem is studied and a priori estimates of the approximate solutions are established, which are useful for proving our existence and regularity results.}

\textbf{Notations:} For the rest of the paper, unless otherwise mentioned, we will use the following notations and assumptions:
\begin{itemize}
    \item For $k,s\in\R$, we define $T_k(s)=\max\{ -k,\, \min\{s,k\}\}$ and $G_k(s)=(|s|-k)^+sgn(s).$
    \item For a measurable set $A\subset\mathbb{R}^N$, $|A|$ denotes the Lebesgue measure of $A$. Moreover, for a function $u:A\to\R$, we define $u^+:=\max\{ u, 0\}$ and $u^-:=\max\{-u, 0\}.$
{
    \item For $\sigma>1$, we define the conjugate exponent of $\sigma$ by $\sigma'=\frac{\sigma}{\sigma-1}$.
}    
    \item $C$ denotes a positive constant, whose value may change from line to line or even in the same line.
    
    \item For a measurable function $f$ over a measurable set $S$ and given constants $c,d$, we write $c\leq u\leq d$ in $S$ to mean that $c\leq u\leq d$ a.e. in $S$.
    
   \item $\Omega\subset\mathbb{R}^N$ with $N\geq 2$ be a bounded Lipschitz domain.

   \item For open sets $\omega$ and $\Omega$ of $\mathbb{R}^N$, $N\geq 2$ by the notation $\omega\Subset\Omega$, we mean that $\overline{\omega}$ is a compact subset of $\Omega$.
\end{itemize}

\section{Functional setting and main results}
{The Sobolev space $W^{1,p}(\Omega)$ for $1<p<\infty$, is defined to be the space of functions $u:\Omega\to\mathbb{R}$ in $L^p(\Omega)$ such that the partial derivatives $\frac{\partial u}{\partial x_i}$ for $1\leq i\leq N$ exist in the weak sense and belong to $L^p(\Omega)$. The space $W^{1,p}(\Omega)$ is a Banach space (see \cite{LC}) equipped with the norm:
$$
\|u\|_{W^{1,p}(\Omega)} = \|u\|_{L^p(\Omega)} + \|\nabla u\|_{L^p(\Omega)},
$$
where $\nabla u=\Big(\frac{\partial u}{\partial x_1},\ldots,\frac{\partial u}{\partial x_N}\Big)$.  
}The fractional Sobolev space $W^{s,p}(\Omega)$ for $0<s<1<p<\infty$, is defined by
$$
W^{s,p}(\Omega)=\Bigg\{{u:\Omega\to\mathbb{R}:\,}u\in L^p(\Omega),\,\frac{|u(x)-u(y)|}{|x-y|^{\frac{N}{p}+s}}\in L^p(\Omega\times \Omega)\Bigg\}
$$
under the norm
$$
\|u\|_{W^{s,p}(\Omega)}=\left(\int_{\Omega}|u(x)|^p\,dx+\int_{\Omega}\int_{\Omega}\frac{|u(x)-u(y)|^p}{|x-y|^{N+ps}}\,dx\,dy\right)^\frac{1}{p}.
$$
We refer to \cite{Hitchhikersguide} and the references therein for more details on fractional Sobolev spaces. Due to the mixed behavior of our equations, following \cite{Vecchihong, VecchiBO, Vecchihenon}, we consider the space
$$
W_0^{1,p}(\Omega)=\{u\in W^{1,p}(\mathbb{R}^N):u=0\text{ in }\mathbb{R}^N\setminus\Omega\}
$$
under the norm
$$
\|u\|_{W_0^{1,p}(\Omega)}=\left(\int_{\Omega}|\nabla u|^p\,dx+\int_{\mathbb{R}^{N}}\int_{\mathbb{R}^{N}}\frac{|u(x)-u(y)|^p}{|x-y|^{N+ps}}\, dx dy\right)^\frac{1}{p}.
$$
{Using Lemma \ref{locnon1}, we observe that the norm $\|u\|_{W_0^{1,p}(\Omega)}$ defined above is equivalent to the norm $\|u\|=\|\nabla u\|_{L^p(\Omega)}$.} Let $0<s\leq 1<p<\infty$. Then we say that $u\in W^{s,p}_{\mathrm{loc}}(\Omega)$ if $u\in W^{s,p}(K)$ for every $K\Subset\Omega$.\\
{We define Marcinkiewicz space $M^q(\Omega)$ as the set of all measurable functions $u:\Omega\to\R$ such that there exists $C>0$,
$$|\{x\in\Omega: |u(x)|>t\}|\leq \frac{C}{t^q}, \text{ for all } t>0.$$
Note that for a bounded domain $\Omega$, it is enough to have this inequality for all $t\geq t_0$ for some $t_0>0.$ The following embeddings are continuous
\begin{align}\label{maremb}
    L^q(\Omega)\hookrightarrow M^q(\Omega)\hookrightarrow L^{q-\eta}(\Omega),
\end{align}
for any {$\eta\in(0,q-1].$}} For more details, see \cite{JMar} and the references therein.\\
For the next result, we refer to \cite[Proposition 2.2]{Hitchhikersguide}.
\begin{Lemma}\label{l1}
Let $0<s<1<p<\infty$. Then there
exists a positive constant $C=C(N,p,s)$ such that $$
\|u\|_{W^{s,p}(\Omega)}\leq C\|u\|_{W^{1,p}(\Omega)}
$$
for every $u\in W^{1,p}(\Omega)$.
\end{Lemma}
The following result {is taken from \cite[Lemma $2.1$]{Silva}, which follows from Lemma \ref{l1} above.}
\begin{Lemma}\label{locnon1}
Let $0<s<1<p<\infty$. There exists a constant $C=C(N,p,s,\Omega)>0$ such that
\begin{equation}\label{locnonsem}
\int_{\mathbb{R}^N}\int_{\mathbb{R}^N}\frac{|u(x)-u(y)|^p}{|x-y|^{N+ps}}\,dx\,dy\leq C\int_{\Omega}|\nabla u|^p\,dx
\end{equation}
for every $u\in W_0^{1,p}(\Omega)$.
\end{Lemma}

For {the subsequent Sobolev embedding, refer to \cite{LC}, for instance}.
\begin{Lemma}\label{emb}
Let $1<p<\infty$. Then the embedding operators
\[
W_0^{1,p}(\Omega)\hookrightarrow
\begin{cases}
L^t(\Omega),&\text{ for }t\in[1,p^{*}],\text{ if }1<p<N,\\
L^t(\Omega),&\text{ for }t\in[1,\infty),\text{ if }p=N,\\
L^\infty(\Omega),&\text{ if }p>N
\end{cases}
\]
are continuous. Moreover, they are compact except for $t=p^*$ if $1<p<N$. Here $p^*=\frac{Np}{N-p}$ if $1<p<N$.
\end{Lemma}

Along the lines of the proof of \cite[Proposition 2.3]{Canino}, the result stated below holds.
\begin{Lemma}\label{Prop2.3}
Let $0<s<1<p<\infty$ and $u\in W^{s,p}_{\mathrm{loc}}(\Omega)\cap L^1(\Omega)$ and $u=0$ a.e. in $\mathbb{R}^N\setminus\Omega$. Then for any $\phi\in C_c^{1}(\Omega)$, we have
$$
\int_{\mathbb{R}^N}\int_{\mathbb{R}^N}{(u(x)-u(y))(\phi(x)-\phi(y))}K(x,y)\,d x dy<\infty.
$$
\end{Lemma}
{The following result from \cite[Theorem $2.1$]{OPdie} will be useful to prove our existence theorems.
\begin{Theorem}\label{ref}
    Suppose $\{f_n\}_{n\in\mathbb{N}}$ be a sequence in $L^1(\Omega)$ such that $f_n\rightharpoonup f$ weakly in $L^1(\Omega)$ and $\{g_n\}_{n\in\mathbb{N}}$ be a sequence in $L^\infty(\Omega)$ such that $g_n$ converges to g in a.e. in $\Omega$ and $\mathrm{weak^*}$ in $L^\infty(\Omega).$ Then $$\lim_{n\to\infty}\int_\Omega f_ng_n\, dx=\int_\Omega fg\, dx.$$
\end{Theorem}
}
Next, we mention some preliminary results related to measures (see \cite{GOB, MD, OPdie}). {We define $M(\Omega)$ as a set of all signed Radon measures on $\Omega$ with bounded total variation (as usual, identified with a linear map $u\to \int_\Omega u\, d\mu$ on $C(\overline{\Omega})$). If $\nu\in M(\Omega)$ is a non-negative Radon measure then by the Lebesgue's decomposition theorem \cite[page 384]{Royden}, we have $$\nu=\nu_a+\nu_s,$$  where $\nu_a$ is absolutely continuous with respect to the Lebesgue measure and $\nu_s$ is singular with respect to the Lebesgue measure. By the Radon-Nikodym theorem \cite[page 382]{Royden}, there exists a non-negative Lebesgue measurable function $f$ such that for every measurable set $E\subset\Omega$,
$$\nu_a(E)=\int_E f\, dx.$$ Furthermore, if $\nu$ is bounded then $f\in L^1(\Omega)$. If the function $f$ is not identically zero function, then we say that $\nu$ is {non-singular} with respect to the Lebesgue measure, otherwise it is called {purely} singular measure.

Let us now review the definition of $p$-capacity, which will help us to characterize the data in our problem (see \cite{Heinonen}).} Suppose $p>1$, then for a compact set $K\subset\Omega$, the $p$-capacity of $K$ is denoted by $\mathrm{cap_{\text{$p$}}(K)}$ and defined as
$$\mathrm{cap\text{$_p$}(K)}:=\inf \left \{ \int_\Omega|\nabla\phi|^p dx : \phi\in C^\infty_0(\Omega), \phi\geq \chi_K \right \},$$
where $$\chi_K(x):=\begin{cases}
    1\text{ if } x\in K,\\
    0 \text{ otherwise.}
\end{cases}$$
For an open set $U\subset\Omega$, the $p$-capacity is defined by
$$\mathrm{cap_{\text{$p$}}(U)}:=\sup\left\{\mathrm{cap_{\text{$p$}}(K)}: K\subset U \text{ is compact}\right\}.$$
Finally, the $p$-capacity of any subset $B$ of $\Omega$ is defined by
$$\mathrm{cap_{\text{$p$}}(B)}:=\inf\left \{\mathrm{cap_{\text{$p$}}(U)}: U \text{ is an open set in } \Omega \text{ containing } B\right \}.$$
We say a measure $\nu\in M(\Omega)$ is absolutely continuous with respect to $p$-capacity if the following holds: \textit{$\nu(E)=0$ for every $E\subset\Omega$ such that $\mathrm{cap_\text{$p$}(E)}=0$.}
We define 
$$
M^p_0(\Omega)=\{\nu\in M(\Omega):\nu \text{ is absolutely continuous with respect to } p \text{-capacity}\}.
$$
One can observe that if $1<p_1<p_2$, then $M^{p_1}_0(\Omega)\subset M^{p_2}_0(\Omega).$

The following characterization from \cite[Theorem 2.1]{GOB} is very useful for us.
\begin{Theorem}\label{dec1}
    Let $1<p<\infty$ and $\nu\in M^p_0(\Omega)$. Then there exists $f\in L^1(\Omega)$ and $G\in (L^{p'}(\Omega))^N$ such that $\nu=f-\mathrm{div}(G)\in L^1(\Omega)+W^{-1,p'}(\Omega)$ {in $\mathcal{D}'(\Omega)$ (space of distributions).} Furthermore, if $\nu$ is non-negative then $f$ is non-negative.
\end{Theorem}}
Next, we define the notion of weak solutions of the problem \eqref{ME}.
\begin{Definition}\label{def1}
    Let $0<s<1<q<\infty$ and $\delta:\overline{\Omega}\to(0,\infty)$ be a continuous function. {Suppose that $\mu$ and $\nu$ are two non-negative bounded Radon measures on $\Omega$ such that $\nu\in M^q_0(\Omega)$. We say that $u\in W^{1,q}_{loc}(\Omega)\cap L^1(\Omega)$ is a weak solution of the equation (\ref{ME}) if $u=0$ in $\mathbb{R}^N\setminus\Omega$ and 
    \begin{enumerate}
        \item[(a)] for every $\omega\Subset\Omega$, there exists a constant $C(\omega)>0$ such that $u\geq C(\omega)>0$ in $\omega$ and
        
        \item[(b)] for every $\phi\in C^{\infty}_c(\Omega)$, we have
    \begin{align}\label{SC}
          \int_\Omega A(x)\nabla u\cdot\nabla\phi\, dx+\int_{\mathbb{R}^N}\int_{\mathbb{R}^N}&{(u(x)-u(y))(\phi(x)-\phi(y))}K(x,y)\, dx dy\nonumber\\
          &=\int_\Omega\frac{\phi}{u^{\delta(x)}}\, d\nu+\int_\Omega \phi\,  d\mu.
    \end{align}
     \end{enumerate}
    }
\end{Definition}

\begin{Remark}
We remark that Definition \ref{def1} is well stated. More precisely, if $\phi\in C^{\infty}_c(\Omega)$ and $K=\mathrm{supp}\,\phi$, since $|A(x)|\leq \beta$, we have
$$
\left|\int_\Omega A(x)\nabla u\cdot\nabla\phi\,dx\right|\leq \beta\|\phi\|_{L^\infty(\Omega)}\|{\nabla} u\|_{L^1(K)}<\infty.
$$
Moreover, combining Lemma \ref{l1} and Lemma \ref{Prop2.3}, it follows that
$$
\left|\int_{\mathbb{R}^N}\int_{\mathbb{R}^N}{(u(x)-u(y))(\phi(x)-\phi(y))}K(x,y)\,d x dy\right|<\infty.
$$
{Since $\phi\in C^\infty_c(\Omega)$ and $\mu$ is a non-negative Radon measure, therefore 
$$\int_\Omega \phi d\mu<\infty.$$ Furthermore, as in \cite[Remark 3.2]{OPdie}, since $\nu\in M^q_0(\Omega)$, so $\nu\in L^1(\Omega)+W^{-1,q'}(\Omega)$ (see Lemma \ref{dec1}) and due to the above property (a) along with that $\phi\in C^{\infty}_c(\Omega),$ we have $\frac{\phi}{u^{\delta(x)}}\in W^{1,q}_0(\Omega)\cap L^\infty(\Omega).$ By keeping this fact in mind, with a little abuse of notation we denote
$$\int_\Omega\frac{\phi}{u^{\delta(x)}}d\nu:=\Big\langle\nu, \frac{\phi}{u^{\delta(x)}}\Big\rangle_{L^1(\Omega)+W^{-1,q'}(\Omega), W^{1,q}_0(\Omega)\cap L^\infty(\Omega)}.$$}
\end{Remark}

The following approximation result is very useful for our justification \cite{GOB, PM, OPdie}.
\begin{Lemma}\label{aprox}
Let $\nu=f-\mathrm{div}(G)$ be a non-negative bounded Radon measure in $M^p_0(\Omega)$ {for some $1<p<\infty$}, where $f\in L^1(\Omega)$ and $G\in (L^{p'}(\Omega))^N.$ Then there exists a sequence of non-negative functions {$\{\nu_n\}_{n\in\mathbb{N}}\in L^2(\Omega)$} in $\Omega$ such that 
\begin{enumerate}
    \item $\nu_n=f_n-\mathrm{div}(G_n)$ in $D'(\Omega)$ and
    \item {$\{\nu_n\}_{n\in\mathbb{N}}$} is uniformly bounded in $L^1(\Omega),$
\end{enumerate}
where $f_n\in L^2(\Omega)$ such that $f_n\rightharpoonup f$ {weakly} in $L^1(\Omega)$ and $G_n\to G$ {strongly} in $(L^{p'}(\Omega))^N.$
\end{Lemma}
\begin{Definition}
A sequence {$\{\mu_n\}_{n\in\mathbb{N}}\subset M(\Omega)$} is said to converge to a measure $\mu\in M(\Omega)$ in narrow topology if for every $\phi\in C^\infty_c(\Omega)$, we have $$\lim_{n\to\infty}\int_\Omega \phi d\mu_n=\int_\Omega \phi d\mu.$$
\end{Definition}
Before stating our main results below, we define the condition $(P_{\epsilon,\delta_*})$ below.\\
\textbf{Condition $(P_{\epsilon,\delta_*})$:} We say that a continuous function $\delta:\overline{\Omega}\to (0,\infty)$, satisfies the condition $(P_{\epsilon,\delta_*})$, if there exist $\delta_*\geq 1$ and $\epsilon>0$ such that $\delta(x)\leq \delta_*$ for every $x\in\Omega_\epsilon$, where 
$
\Omega_\epsilon:=\{y\in\Omega:\text{dist\,}(y,\partial\Omega)<\epsilon\}.
${
\begin{Remark}\label{rmkdel}
We observe that, if a continuous function $\delta:\overline{\Omega}\to (0,\infty)$ satisfies the condition $(P_{\epsilon,\delta_*})$ for some $\delta_*\geq 1$ and $\epsilon>0$, then for every $\gamma\geq \delta_*$, the function $\delta$ satisfies the condition $(P_{\epsilon,\gamma})$. Therefore, $\delta_*$ is not uniquely determined by $\delta$. Moreover, by the continuity of $\delta$, it follows that $\max_{\partial\Omega}\,\delta\leq\delta_*$. Furthermore, if $M>\max_{\partial\Omega}\,\delta$, then from \cite[page 493]{CMvar}, it follows that there exists an $\epsilon>0$ such that $\delta\leq M$ in $\Omega_{\epsilon}$ and hence, one can choose $\delta_*=\max\{1,M\}$.
\end{Remark} 
We state our main results only for the case 
$N>2$. However, we would like to emphasize that analogous results also hold for $N=2$, by taking into account Lemma \ref{emb} and following the lines of proof as those for the main results. More precisely, for $N=2$, exact statement for Theorems \ref{T3}-\ref{T1} will be valid and analogous statements will hold for remaining main results.}

First, we state our main existence results which reads as follows:
\subsection{Existence results}
\begin{Theorem}\label{T3}(Variable singular exponent)
    Let $\delta:\overline{\Omega}\to(0,\infty)$ be a continuous function {which is locally Lipschitz continuous in $\Omega$} and satisfies the condition $(P_{\epsilon,\delta_*})$ for some $\delta_*\geq 1$ and {for some }$\epsilon>0$. Let $1<p<\frac{N}{N-1}$ and suppose that $\nu,\,\mu$ are non-negative bounded Radon measures on $\Omega$ such that $\nu\in M^{p}_0(\Omega)$ and $\nu$ is non-singular with respect to the Lebesgue measure. Then the problem $(\ref{ME})$ admits a weak solution $u\in W^{1,p}_{\mathrm{loc}}(\Omega)\cap L^1(\Omega)$ in the sense of Definition \ref{def1} such that 
    \begin{enumerate}
        \item[(i)] If $\delta_*=1$, then $u\in W^{1,p}_0(\Omega)$.
        \item[(ii)] If $\delta_*>1$, then $u\in W^{1,p}_{\mathrm{loc}}(\Omega)$ such that $T_k(u)\in W^{1,2}_{\mathrm{loc}}(\Omega)$ and $T_k^\frac{\delta_*+1}{2}(u)\in W^{1,2}_0(\Omega)$ for every $k>0.$
    \end{enumerate}
\end{Theorem}
If $\nu$ turns out to be an integrable function, then we do not require {$\delta$ to be locally Lipschitz continuous in $\Omega$.} In this case the result is stated as follows:
\begin{Theorem}\label{T1}(Variable singular exponent)
Let $\delta:\overline{\Omega}\to (0,\infty)$ be a continuous function satisfying the condition $(P_{\epsilon,\delta_*})$ for some $\delta_*\geq 1$ and {for some} $\epsilon>0$. Assume that $\nu\in L^1(\Omega)\setminus\{0\}$ is a non-negative function in $\Omega$ and $\mu$ is a non-negative bounded Radon measure on $\Omega.$ Then for every $1<p<\frac{N}{N-1}$, the equation (\ref{ME}) admits a weak solution $u\in W^{1,p}_{\mathrm{loc}}(\Omega)\cap L^1(\Omega)$ in the sense of the Definition \ref{def1} such that
 \begin{enumerate}
        \item[(i)] If $\delta_*=1$, then $u\in W^{1,p}_0(\Omega)$.
        \item[(ii)] If $\delta_*>1$, then $u\in W^{1,p}_{\mathrm{loc}}(\Omega)$ {such that $T_k(u)\in W^{1,2}_{\mathrm{loc}}(\Omega)$ and }$T_k^\frac{\delta_*+1}{2}(u)\in W^{1,2}_0(\Omega)$ for every $k>0.$
    \end{enumerate}
\end{Theorem}

\begin{Remark}\label{rmkT1-1}
Theorem \ref{T1} extends \cite[Theorem 1.1]{Ghosh} to the variable exponent case and to more general class of mixed operators $\mathcal{M}$.
\end{Remark}
{
\begin{Remark}\label{reg}
We observe that, under the hypotheses in Theorems \ref{T3} and \ref{T1}, for every $\gamma\geq\delta_*$ and $k>0$, the function $T_k(u)^\frac{\gamma+1}{2}$ belongs to $W_0^{1,2}(\Omega)$ in Theorems \ref{T3} and \ref{T1}.
\end{Remark}}
Our next result tells that when $\mu$ turns out to be a function and the function $\delta$ is constant, then we can relax the condition on $\nu$ to obtain the existence of a solution. 
\begin{Theorem}\label{T2}(Constant singular exponent)
    Assume that $\delta:\overline{\Omega}\to(0,\infty)$ be a constant function. We define
    $$q=\begin{cases}
    \frac{N(\delta+1)}{N+\delta-1}, \text{ if }0<\delta<1,\\
    2, \text{ if }\delta\geq 1.
\end{cases}$$
{Let $\mu\in L^\frac{N(\delta+1)}{N+2\delta}(\Omega )$ be a non-negative function in $\Omega$. Further, assume that $\nu\in M^q_0(\Omega)$ is a non-negative bounded Radon measure on $\Omega$ such that $\nu$ is non-singular with respect to the Lebesgue measure.} Then the equation (\ref{ME}) admits a weak solution $u\in W^{1,q}_{\mathrm{loc}}(\Omega)\cap L^1(\Omega)$ in the sense of Definition \ref{def1} such that
\begin{enumerate}
    \item [(i)] If $0<\delta\leq 1$, then $u\in W_0^{1,q}(\Omega)$.
    \item[(ii)] If $\delta>1$, then $u\in W_{\mathrm{loc}}^{1,2}(\Omega)$ such that $u^\frac{\delta+1}{2}\in W_0^{1,2}(\Omega)$. 
\end{enumerate}
\end{Theorem}
\begin{Remark}\label{mrmk}
One can observe along the lines of the proofs of our main existence results that Theorems \ref{T3}-\ref{T2} holds even for the purely local equation that can be obtained by replacing $\mathcal{M}$ with the operator $\mathcal{A}$ in the equation \eqref{ME} where $\mathcal{A}u=\text{div}(A(x)\nabla u)$ with $A:\Omega\to\mathbb{R}^{N^2}$ a bounded elliptic matrix satisfying \eqref{lkernel}. To the best of our knowledge, such results are new even in the purely local case.
\end{Remark}
Our main regularity results are stated below.
\subsection{Regularity results}
\begin{Theorem}\label{T4}(Constant singular exponent)
    {Let $\delta:\overline{\Omega}\to [1,\infty)$ be a constant function and $\nu\in L^r(\Omega)\setminus\{0\}$, $\mu\in L^m(\Omega)$ for some $r,m\geq 1$ be two non-negative functions in $\Omega$.} Suppose $u$ is the weak solution of the problem (\ref{ME}) obtained in {Theorem \ref{T1}.} Then the following conclusions hold:
    \begin{enumerate}
        \item[(i)] If $r>\frac{N}{2}$, $m>\frac{N}{2}$, then $u\in L^\infty(\Omega)$.
        \item[(ii)] If $r>\frac{N}{2},$ $1<m<\frac{N}{2}$, then $u\in L^{m^{**}}(\Omega)$, {where $m^{**}=\frac{Nm}{N-2m}$.}
        \item[(iii)] If $1\leq r<\frac{N}{2}$, $m>\frac{N}{2}$, then $u\in L^{\frac{Nr(\delta+1)}{N-2r}}(\Omega)$.
        \item[(iv)] If $1\leq r<\frac{N}{2}$, $1<m<\frac{N}{2}$, then $u\in L^{q_1}(\Omega)$, 
        where $q_1=\min\Big\{m^{**},\frac{Nr(\delta+1)}{N-2r}\Big\}$ {with $m^{**}=\frac{Nm}{N-2m}$.} 
    \end{enumerate}
\end{Theorem}

\begin{Theorem}\label{T6}(Constant singular exponent)
{Let $\delta:\overline{\Omega}\to (0,1)$ be a constant function and $\nu\in L^r(\Omega)\setminus\{0\}$, $\mu\in L^m(\Omega)$ for some $r,m\geq 1$ be two non-negative functions in $\Omega$.} Suppose $u$ is the weak solution of the problem (\ref{ME}) obtained in {Theorem \ref{T1}.} Then the following conclusions hold:
    \begin{enumerate}
        \item[(i)] If $r>\frac{N}{2}$, $m>\frac{N}{2}$, then $u\in L^\infty(\Omega)$.
        \item[(ii)] If $r>\frac{N}{2},$ $1<m<\frac{N}{2}$, then $u\in L^{m^{**}}(\Omega)$, {where $m^{**}=\frac{Nm}{N-2m}$.}
        \item[(iii)] If $\left (\frac{2^*}{(1-\delta)}\right )'\leq r<\frac{N}{2}$, $m>\frac{N}{2}$, then $u\in L^{\frac{Nr(\delta+1)}{N-2r}}(\Omega)$.
        \item[(iv)] If $\left (\frac{2^*}{(1-\delta)}\right )'\leq r<\frac{N}{2}$, $1<m<\frac{N}{2}$, then $u\in L^{q_2}(\Omega)$, where $q_2=\min\Big\{m^{**},\frac{Nr(\delta+1)}{N-2r}\Big\}$ {with $m^{**}=\frac{Nm}{N-2m}$.}
    \end{enumerate}
\end{Theorem}

\begin{Theorem}\label{T7}(Variable singular exponent)
 Let $\delta:\overline{\Omega}\to(0,\infty)$ be a continuous function and satisfying the condition $(P_{\epsilon,\delta_*})$ for some $\delta_*\geq 1$ and {for some }$\epsilon>0$. Assume that $\nu\in L^r(\Omega)\setminus\{0\}$ and {$\mu\in L^m(\Omega)$} for some $r,m\geq 1$ are non-negative functions in $\Omega$. Suppose $u$ is the weak solution of the problem (\ref{ME}) obtained in {Theorem \ref{T1}.} Then the following conclusions hold:
    \begin{enumerate}
        \item[(i)] If $r>\frac{N}{2}$, $m>\frac{N}{2}$, then $u\in L^\infty(\Omega)$.
        \item[(ii)] If $r>\frac{N}{2},$ $1<m<\frac{N}{2}$, then $u\in L^{m^{**}}(\Omega)$, {where $m^{**}=\frac{Nm}{N-2m}$.}
        \item[(iii)] If $\frac{N(\delta_*+1)}{N+2\delta_*}\leq r<\frac{N}{2}$, $m>\frac{N}{2}$, then $u\in L^{r^{**}}(\Omega)$, where $r^{**}=\frac{Nr}{N-2r}$.
        \item[(iv)] If $\frac{N(\delta_*+1)}{N+2\delta_*}\leq r<\frac{N}{2}$, $1<m<\frac{N}{2}$, then $u\in L^{q_3}(\Omega)$, where $q_3=\min\big\{m^{**},r^{**}\big\}$ { with $m^{**}=\frac{Nm}{N-2m}$ and $r^{**}=\frac{Nr}{N-2r}$.}
    \end{enumerate}
\end{Theorem}

\section{Proof of existence results}
\subsection{Proof of {Theorem \ref{T3}}}
{Suppose} that $\delta:\overline{\Omega}\to (0,\infty)$ is a {continuous} function {which is locally Lipschitz continuous in $\Omega$} and satisfies the condition $(P_{\epsilon,\delta_*})$ for some $\delta_*\geq 1$ and for some $\epsilon>0$. We prove the result when $\delta_*=1$ and when $\delta_*>1$ separately.

{{Assume that} $\nu\in M^p_0(\Omega)$ for some $1<p<\frac{N}{N-1}$, is a non-singular, non-negative bounded Radon measure on $\Omega$ and $\mu$ is a non-negative bounded Radon measure on $\Omega$.} Due to Lebesgue's decomposition theorem \cite[page 384]{Royden}, $\nu=\nu_a+\nu_s,$ where $\nu_a$ is absolutely continuous with respect to the Lebesgue measure and $\nu_s$ is singular with respect to the Lebesgue measure. By the Radon-Nikodym theorem \cite[ page 382]{Royden}, there exists a non-negative $f\in L^1(\Omega)\setminus\{0\}$ such that for every measurable set $E\subset\Omega$,
$\nu_a(E)=\int_E f\, dx.$

Since $\nu\in M^p_0(\Omega)$ and $0\leq\nu_s\leq\nu$, we have $\nu_s\in M^p_0(\Omega)$. By Theorem \ref{dec1}, there exists $0\leq H\in L^1(\Omega)$ and $G\in (L^{p'}(\Omega))^N$ such that $\nu_s=H-\mathrm{div}\, G$ {{(in distributional sense)}.} Furthermore, there exists a sequence of non-negative functions $\{h_n\}_{n\in\mathbb{N}}\subset L^2(\Omega)$ in $\Omega$ such that $||h_n||_{L^1(\Omega)}\leq C$ for some constant $C>0$ independent of $n$ and $h_n=H_n-\mathrm{div}\, G_n$ in $D'(\Omega)$, where {$H_n\in L^2(\Omega)$ such that} $H_n\rightharpoonup H$ {weakly} in $L^1(\Omega)$ and $G_n\to G$ {strongly} in $(L^{p'}(\Omega))^N$ (see Lemma \ref{aprox}). Since $\mu$ is a non-negative bounded Radon measure on $\Omega$, there exists a {non-negative} sequence $\{g_n\}_{n\in\mathbb{N}}\subset L^\infty(\Omega)$ such that $||g_n||_{L^1(\Omega)}\leq C$ {for some constant $C>0$ independent of $n$} and $g_n\rightharpoonup \mu$ in the narrow topology (\cite[Theorem A.7]{OLP}). For each $n\in\mathbb{N}$, we consider the following approximation of the given equation (\ref{ME}):
\begin{align}\label{DP}
    \begin{cases}
        \mathcal{M}u=\frac{T_n(f)+h_n}{(u+\frac{1}{n})^{\delta(x)}}+{g_n} \text{ in } \Omega,\\
        u=0 \text{ in }\mathbb{R}^N\setminus\Omega \text{ and } u>0 \text{ in }\Omega.
    \end{cases}
\end{align}
    By Lemma \ref{LD11}, for each $n\in\mathbb{N}$, there exists a unique weak solution $u_n\in W^{1,2}_0(\Omega)$ to the equation (\ref{DP}) such that for every $\omega\Subset\Omega$, there exists a constant $C(\omega)>0$ {(independent of $n$)} such that $u_n\geq C(\omega)$ in $\omega$ for all $n$. {By the weak formulation of (\ref{DP}), for every $\phi\in C^\infty_c(\Omega)$, we get
\begin{align}\label{A16}
\int_\Omega &A(x)\nabla u_n\cdot\nabla\phi \, dx+\int_{\mathbb{R}^N}\int_{\mathbb{R}^N}(u_n(x)-u_n(y))(\phi(x)-\phi(y))K(x,y)\, dx dy\nonumber\\
&=\int_\Omega\frac{T_n(f)\phi}{(u_n+\frac{1}{n})^{\delta(x)}}\, dx
+\int_\Omega \frac{H_n\phi}{(u_n+\frac{1}{n})^{\delta(x)}}\, dx+\int_\Omega G_n\cdot\nabla\left (\frac{\phi}{(u_n+\frac{1}{n})^{\delta(x)}}\right )\, dx\nonumber\\
&\quad +\int_\Omega\phi g_n\, dx.
\end{align}
We pass to the limit in \eqref{A16} for the cases $\delta_*=1$ and $\delta_*>1$ below.}
    \begin{enumerate}
\item[$(i)$] {Let $\delta_*=1$. By Lemma \ref{LD2}-$(a)$, there exists a subsequence of $\{u_n\}_{n\in\mathbb{N}}$, still denoted by {$\{u_n\}_{n\in\mathbb{N}}\subset W_0^{1,p}(\Omega)$} and $u\in W^{1,p}_0(\Omega)$} such that $u_n\rightharpoonup u$ {weakly} in $W^{1,p}_0(\Omega)$ and $u_n\to u$ {pointwise} a.e. in $\Omega.$
 Since $u_n\rightharpoonup u$ weakly in  $W^{1,p}_0(\Omega)$, for every $\phi\in C_c^{\infty}(\Omega)$, it follows that 
 \begin{align}\label{A17}
    &\lim_{n\to\infty}\int_\Omega A(x)\nabla u_n\cdot\nabla\phi\, dx=\int_\Omega A(x)\nabla u\cdot\nabla\phi\, dx,
    \end{align}
    and 
    \begin{equation}\label{B00}
    \begin{split}
&\lim_{n\to\infty }\int_{\mathbb{R}^N}\int_{\mathbb{R}^N}(u_n(x)-u_n(y))(\phi(x)-\phi(y))K(x,y)\, dx dy\\
&=\int_{\mathbb{R}^N}\int_{\mathbb{R}^N}(u(x)-u(y))(\phi(x)-\phi(y))K(x,y)\, dx dy.
\end{split}
\end{equation}
Moreover, since $g_n\rightharpoonup\mu$ in {the narrow topology}, for every $\phi\in C_c^{\infty}(\Omega)$, we have
\begin{align}
\lim_{n\to\infty}\int_\Omega \phi g_n\, dx=\int_\Omega \phi d\mu.
\end{align}
{Let $\omega=\mathrm{supp}\,\phi$}, then there exists a constant $C=C(\omega)>0$ {(independent of $n$)} such that $u_n\geq C$ in $\omega$, for all $n$. Since $f\in L^1(\Omega)$, from the Lebesgue's dominated convergence theorem, we get
\begin{align}\label{A18}
    \lim_{n\to\infty} \int_\Omega\frac{T_n(f)\phi}{(u_n+\frac{1}{n})^{\delta(x)}}\, dx=\int_\Omega\frac{f \phi}{u^{\delta(x)}}\, dx\quad\forall\phi\in C_c^{\infty}(\Omega).
\end{align}
{Since $H_n\rightharpoonup H$ weakly in $L^1(\Omega)$ and $\frac{\phi}{(u_n+\frac{1}{n})^{\delta(x)}}\to \frac{\phi}{u^{\delta(x)}}$ pointwise a.e. in $\Omega$ and weak* in $L^\infty(\Omega)$, by Theorem \ref{ref}, it follows that} 
\begin{align}\label{AA2}
    \lim_{n\to\infty}\int_\Omega\frac{H_n\phi}{(u_n+\frac{1}{n})^{\delta(x)}}\, dx=\int_\Omega\frac{H\phi}{u^{\delta(x)}}\, dx.
\end{align}
We only have to pass the limit in the second last term of (\ref{A16}).
We observe that for every $\phi\in C_c^{\infty}(\Omega)$,
\begin{align}\label{SA}
\int_\Omega G_n\cdot\nabla\left (\frac{\phi}{(u_n+\frac{1}{n})^{\delta(x)}}\right )\, dx&=\int_\Omega\frac{G_n\cdot\nabla\phi}{(u_n+\frac{1}{n})^{\delta(x)}}\, dx-\int_\Omega\frac{G_n\cdot\nabla\delta(x)}{(u_n+\frac{1}{n})^{\delta(x)}}\log{\Big(u_n+\frac{1}{n}\Big)}\phi\, dx\nonumber\\
&-\int_\Omega\frac{\delta(x)G_n\cdot\nabla u_n}{(u_n+\frac{1}{n})^{\delta(x)+1}}\phi\, dx.
\end{align}
The facts $G_n\to G$ {strongly} in $(L^{p'}(\Omega))^N$ and $\frac{\nabla\phi}{(u_n+\frac{1}{n})^{\delta(x)}}\to \frac{\nabla\phi}{u^{\delta(x)}}$ {strongly} in $(L^p(\Omega))^N$ are used to deduce that 
\begin{align}\label{SA1}
\lim_{n\to\infty}\int_\Omega\frac{G_n\cdot\nabla\phi}{(u_n+\frac{1}{n})^{\delta(x)}}\, dx=\int_\Omega\frac{G\cdot\nabla\phi}{u^{\delta(x)}}\, dx.
\end{align}
Moreover, since $\frac{\delta(x)\phi\nabla u_n}{(u_n+\frac{1}{n})^{\delta(x)+1}}\rightharpoonup \frac{\delta(x)\phi\nabla u}{u^{\delta(x)+1}}$ weakly in $(L^p(\Omega))^N$, we obtain
\begin{align}\label{Z1}
\lim_{n\to\infty}\int_\Omega\frac{\delta(x)G_n\cdot\nabla u_n}{(u_n+\frac{1}{n})^{\delta(x)+1}}\phi\, dx=\int_\Omega\frac{\delta(x)G\cdot\nabla u}{u^{\delta(x)+1}}\phi\, dx.
\end{align}
In order to pass to the limit in the second last integral in (\ref{SA}), {we use the local Lipschitz continuity of $\delta$ in $\Omega$. To be more precise, we use the fact {$|\nabla\delta|\in L^\infty(\omega)$} to apply Lebesgue's dominated convergence theorem, where $\omega=\mathrm{supp}\,\phi$.} Define $r:=\min_{\overline{\Omega}}\delta(x)>0$ and observe that  $\frac{\log{x}}{x^r}$ is bounded on $[C,\infty)$, where $C>0$ is a uniform lower bound of $u_n$ in $\omega$. We notice that
\begin{align*}
    \left |\frac{G_n\cdot\nabla\delta(x)}{(u_n+\frac{1}{n})^{\delta(x)}}\log{\Big(u_n+\frac{1}{n}\Big)}\phi\right |&=\left |\left (\frac{\phi}{(u_n+\frac{1}{n})^{\delta(x)-r}}\right )\left (\frac{\log{\Big(u_n+\frac{1}{n}\Big)}}{(u_n+\frac{1}{n})^r}\right ) G_n\cdot\nabla\delta\right|\leq C|G_n|,
\end{align*}
 in $\omega$. Using this together with the fact $\lim_{n\to\infty} \int_\Omega |G_n|=\int_\Omega |G|$, we conclude
\begin{align}\label{AA3}
\lim_{n\to\infty}\int_\Omega\frac{G_n\cdot\nabla\delta(x)}{(u_n+\frac{1}{n})^{\delta(x)}}\log{\Big(u_n+\frac{1}{n}\Big)}\phi\, dx=\int_\Omega \frac{G\cdot\nabla\delta}{u^{\delta(x)}}\log{u}\, \phi dx.
\end{align}
Combining (\ref{SA1})-(\ref{AA3}), the identity (\ref{SA}) leads to
\begin{align}\label{SA2}
    \lim_{n\to\infty}\int_\Omega G_n\cdot\nabla\left (\frac{\phi}{(u_n+\frac{1}{n})^{\delta(x)}}\right )\, dx=\int_\Omega G\cdot\nabla\left (\frac{\phi}{u^{\delta(x)}}\right )\, dx.
\end{align}
Thus, letting  $n\to\infty$ in both sides of the equality (\ref{A16}) and using (\ref{A17})-(\ref{AA2}) and (\ref{SA2}), we obtain
\begin{equation*}
\begin{split}
&\int_\Omega A(x)\nabla u\cdot\nabla\phi\, dx+\int_{\mathbb{R}^N}\int_{\mathbb{R}^N}(u(x)-u(y))(\phi(x)-\phi(y))K(x,y)\, dx dy\\
&=\int_\Omega\frac{f\phi}{u^{\delta(x)}}\, dx+\int_\Omega \frac{H\phi}{u^{\delta(x)}}\, dx\nonumber+\int_\Omega G\cdot\nabla\left (\frac{\phi}{u^{\delta(x)}}\right )\, dx+\int_\Omega\phi d\mu\nonumber\\
&=\int_\Omega \frac{\phi}{u^{\delta(x)}}\, d\nu+\int_\Omega \phi d\mu.
\end{split}
\end{equation*}
Hence, $u\in W^{1,p}_0(\Omega)$ is a weak solution of the equation (\ref{ME}).
\item[$(ii)$]
{Let $\delta_*>1$. We apply Lemma \ref{LD2}-$(b)$ and conclude that there exists a subsequence of $\{u_n\}_{n\in\mathbb{N}}$, still denoted by {$\{u_n\}_{n\in\mathbb{N}}\subset W^{1,p}_{\mathrm{loc}}(\Omega)$} and $u\in W^{1,p}_{\mathrm{loc}}(\Omega)$ such that }
\begin{align*}
    \begin{cases}
        u_n\rightharpoonup u \text{ weakly in } W^{1,p}_{\mathrm{loc}}(\Omega),\\
        u_n\to u \text{ pointwise a.e. in }\Omega.
    \end{cases}
\end{align*}
{In this case by repeating} the similar proof as in $(i)$ above, one can pass to the limit in all the integral of the R.H.S. in {(\ref{A16})} except the second integral, which is nonlocal. In order to pass to the limit there, by Lemma \ref{LD2}, since the sequence $\Big\{T_1^\frac{\delta_*+1}{2}(u_n)\Big\}_{n\in\mathbb{N}}$ is uniformly bounded in $W^{1,2}_0(\Omega)$, by a similar argument as in \cite [Theorem 3.6]{Sciunzi}, we have
\begin{align}\label{TE1}
   \lim_{n\to\infty} \int_{\mathbb{R}^N}&\int_{\mathbb{R}^N}(T_1(u_n)(x)-T_1(u_n)(y))(\phi(x)-\phi(y))K(x,y)\, dxdy=\nonumber\\
   &\int_{\mathbb{R}^N}\int_{\mathbb{R}^N}(T_1(u)(x)-T_1(u)(y))(\phi(x)-\phi(y))K(x,y)\, dxdy.
\end{align}
Furthermore, by Lemma \ref{LD2}, since $\{G_1(u_n)\}_{n\in\mathbb{N}}$ is uniformly bounded in $W^{1,p}_0(\Omega)$, {up to a subsequence we have} $G_1(u_n)\rightharpoonup G_1(u)$ {weakly} in $W^{1,p}_0(\Omega)$. Thus
\begin{align}\label{TE2}
     \lim_{n\to\infty} \int_{\mathbb{R}^N}&\int_{\mathbb{R}^N}(G_1(u_n)(x)-G_1(u_n)(y))(\phi(x)-\phi(y))K(x,y)\, dxdy=\nonumber\\
     &\int_{\mathbb{R}^N}\int_{\mathbb{R}^N}(G_1(u)(x)-G_1(u)(y))(\phi(x)-\phi(y))K(x,y)\, dxdy.
\end{align}
Combining (\ref{TE1}) and (\ref{TE2}) and using the fact $u_n=T_1(u_n)+G_1(u_n)$, we obtain
\begin{align}
     \lim_{n\to\infty} \int_{\mathbb{R}^N}\int_{\mathbb{R}^N}(u_n(x)-u_n(y))&(\phi(x)-\phi(y))K(x,y)\, dxdy=\nonumber\\
     &\int_{\mathbb{R}^N}\int_{\mathbb{R}^N}(u(x)-u(y))(\phi(x)-\phi(y))K(x,y)\, dxdy.
\end{align}
Letting $n\to\infty$ on both sides of the identity {(\ref{A16})}, we obtain
\begin{align*}
     \int_\Omega A(x)\nabla u\cdot\nabla\phi\, dx+\int_{\mathbb{R}^N}\int_{\mathbb{R}^N}&(u(x)-u(y))(\phi(x)-\phi(y))K(x,y)\, dxdy\nonumber\\
     &=\int_\Omega\frac{\phi}{u^\delta}\, d\nu+\int_\Omega \phi d\mu.
\end{align*}
Thus $u\in W^{1,p}_{\mathrm{loc}}(\Omega)$ solves the equation (\ref{ME}). {Moreover, by Lemma \ref{LD2}-$(b)$ and {Lemma \ref{emb}}, one has the sequence $\{G_1(u_n)\}_{n\in\mathbb{N}}$ is uniformly bounded in $L^1(\Omega)$. Taking this into account along with the fact that $|T_1(u_n)|\leq 1$ in $\Omega$ and $u_n=T_1(u_n)+G_1(u_n)$, we obtain that $\{u_n\}_{n\in\mathbb{N}}$ is uniformly bounded in $L^1(\Omega)$. Thus, by Fatou's Lemma, it follows that $u\in L^1(\Omega).$ Furthermore, using Lemma \ref{LD2}-$(b)$ one obtain that $T_k(u)\in W^{1,2}_{\mathrm{loc}}(\Omega)$ such that $T^\frac{\delta_*+1}{2}_k(u)\in W^{1,2}_0(\Omega)$ for every $k>0$.} 
\end{enumerate}

\subsection{Proof of {Theorem \ref{T1}}}
{Suppose that $\delta:\overline{\Omega}\to (0,\infty)$ is a continuous function satisfying the condition $(P_{\epsilon,\delta_*})$ for some $\delta_*\geq 1$ and for some $\epsilon>0$.} Since $\nu\in L^1(\Omega)\setminus\{0\}$ is a non-negative, therefore $\nu\in M^p_0(\Omega)$ for every $1<p<\frac{N}{N-1}$. Moreover, {since} $\nu_s=0$, hence the gradient of $\delta$ does not appear, so the {locally Lipschitz continuity of $\delta$ is not required here}. Thus taking into account Lemma \ref{LD11}  and Lemma \ref{LD2}, the proof follows {along} the {lines of the proof of Theorem \ref{T3}.}
\subsection{Proof of Theorem \ref{T2}} 
{Since $\nu\in M^{q}_0(\Omega)$ is a non-negative bounded Radon measure on $\Omega$, which is non-singular with respect to Lebesgue measure,
using the same arguments as in the proof of Theorem \ref{T3}, we can find a non-negative function $f\in L^1(\Omega)\setminus\{0\}$ and a sequence of non-negative functions $\{h_n\}_{n\in\mathbb{N}}\subset L^2(\Omega)$ with the same properties found in the proof of Theorem \ref{T3} except that the exponent $p$ there will be replaced with the above exponent $q$ here. Taking this into account, for each $n\in\mathbb{N}$, we consider the following approximation of the equation (\ref{ME}):
\begin{align}\label{BP}
    \begin{cases}
    \mathcal{M}u=\frac{T_n(f)+h_n}{(u+\frac{1}{n})^\delta}+T_n(\mu) \text{ in } \Omega,\\
    u=0 \text{ in }\mathbb{R}^N\setminus\Omega \text{ and } u>0 \text{ in }\Omega,
    \end{cases}
\end{align}
where $\mu\in L^\frac{N(\delta+1)}{N+2\delta}(\Omega)$ is a non-negative function in $\Omega$. Finally, we pass to the limit in the weak formulation of \eqref{BP} to conclude the result. This follows by taking into account Lemma \ref{LD11}, Lemma \ref{BL2}-$(a), (b)$ along with Lemma \ref{BL2}-$(c)$ and proceeding along the lines of the proof of Theorem \ref{T3}.}

\section{Proof of regularity results}
\textbf{Proof of Theorem \ref{T4}:}
\begin{enumerate}
\item[$(i)$] If $r,m>\frac{N}{2}$, then by Lemma \ref{EL1}, Lemma \ref{EL2} and Lemma \ref{EL3}, we conclude that $\{u_n\}_{n\in\mathbb{N}}$ is uniformly bounded in $L^\infty(\Omega).$ Since the solution $u,$ obtained in Theorem \ref{T3} is the pointwise limit of $u_n$, we have $u\in L^\infty(\Omega).$

\item[$(ii)$] If $r>\frac{N}{2},\, 1<m<\frac{N}{2}$, then by  Lemma \ref{EL1}, Lemma \ref{EL2} and Lemma \ref{EL3}, we obtain $$\int_\Omega |u_n|^{m^{**}}\, dx\leq C,$$
{for some constant $C>0$ independent of $n$.} Since $u_n\to u$ pointwise a.e. in $\Omega,$ by Fatou's Lemma, we have
$$\int_\Omega |u|^{m^{**}}\, dx\leq \liminf_{n\to\infty}\int_\Omega|u_n|^{m^{**}}\, dx\leq C,$$
for some constant $C>0$ independent of $n$.
Hence, $u\in L^{m^{**}}(\Omega).$
\end{enumerate}
Proofs of part $(iii)$-$(iv)$ follows in a similar way by taking into account Lemma \ref{EL1}, Lemma \ref{EL2} and Lemma \ref{EL3}.\\ 
\textbf{Proof of Theorem \ref{T6}:} The proof follows from Lemma
 \ref{EL1}, Lemma \ref{EL2} and Lemma \ref{EL4}.\\
\textbf{Proof of Theorem \ref{T7}:} The proof follows from Lemma
 \ref{EL1}, Lemma \ref{EL2} and Lemma \ref{EL5}.

\section{Appendix}
\subsection{Approximate problem}
Throughout this subsection, we assume that {$\delta:\overline{\Omega}\to(0,\infty)$} is a continuous function. Let $n\in\mathbb{N}$ and we define $T_n(f(x))=\min\{f(x),n\}$, where $f\in L^1(\Omega)\setminus\{0\}$ is a non-negative function in $\Omega$. Further, we assume that $h_n,\,g_n\in L^2(\Omega)$ are non-negative functions in $\Omega$. Then for each fixed $n\in\mathbb{N}$, we consider the following approximated problem
\begin{align}\label{APE1}
    \begin{cases}
        \mathcal{M}u=\frac{T_n(f)+h_n}{(u+\frac{1}{n})^{\delta(x)}}+g_n \text{ in }\Omega,\\
        u=0 \text{ in }\mathbb{R}^N\setminus\Omega \text{ and } u>0 \text{ in }\Omega.
    \end{cases}
\end{align}

\begin{Lemma}\label{LD11}
    Let $n\in\mathbb{N}$. Then there exists a {unique} weak solution $u_n\in W^{1,2}_0(\Omega)$ of the equation (\ref{APE1}). Moreover, for every $\omega\Subset\Omega$, there exists a constant $C(\omega)>0$ (independent of $n$) such that $u_n\geq C(\omega)$ in $\omega$ for all $n$.
\end{Lemma}
\begin{proof}
\textbf{Existence:} Let $n\in\mathbb{N}$. By the Lax-Milgram theorem {\cite[page 315]{LC}}, we define a map {$\Theta:L^2(\Omega)\to L^2(\Omega)$} by $\Theta v=w$, where $w$ is the unique solution of the following equation:
\begin{align}\label{DP1}
   \begin{cases}
        \mathcal{M}w=\frac{T_n(f)+h_n}{(|v|+\frac{1}{n})^{\delta(x)}}+g_n \text{ in }\Omega,\\
    w=0 \text{ in }\mathbb{R}^N\setminus\Omega.
   \end{cases}
\end{align}
{First, we claim that $\Theta$ is continuous.} {To this end, let $\{v_k\}_{k\in\mathbb{N}}\subset L^2(\Omega)$ be such that $v_k\to v$ strongly in $L^2(\Omega)$. We claim that $w_k:=\Theta v_k$ converges to $w:=\Theta v$ in $L^2(\Omega)$. 
It suffices to show that every subsequence of $\{w_k\}_{k\in\mathbb{N}}$ has a further subsequence that converges to $w$ strongly in $L^2(\Omega).$ Our claim is that $\{w_k\}_{k\in\mathbb{N}}$ has a subsequence that converges to $w$ in $L^2(\Omega)$ and the proof for any other subsequence of $\{w_k\}_{k\in\mathbb{N}}$ is similar. Since $v_k\to v$ strongly in $L^2(\Omega)$, there exists a subsequence of $\{v_k\}_{k\in\mathbb{N}}$, still denoted $\{v_k\}_{k\in\mathbb{N}}$ such that $v_k\to v$ pointwise a.e. in $\Omega.$}
Since the equation (\ref{DP1}) is linear, we have
\begin{align}\label{MeDa}
    \begin{cases}
        \mathcal{M}(w_k-w)=\frac{T_n(f)+h_n}{(|v_k|+\frac{1}{n})^{\delta(x)}}-\frac{T_n(f)+h_n}{(|v|+\frac{1}{n})^{\delta(x)}}\text{ in }\Omega,\\
    w_k-w=0 \text{ in }\mathbb{R}^N\setminus\Omega.
    \end{cases}
\end{align}
We choose $\phi=(w_k-w)$ as test function in the weak formulation of (\ref{MeDa}) and apply (\ref{lkernel}), (\ref{nkernel}) to deduce
\begin{align*}
    \alpha\int_\Omega |\nabla(w_k-w)|^2\, dx&+\Lambda^{-1}\underbrace{\int_{\mathbb{R}^N}\int_{\mathbb{R}^N}\frac{((w_k-w)(x)-(w_k-w)(y))^2}{|x-y|^{N+2s}}\, dxdy}_{\geq 0}\\
    &\leq \int_\Omega \left (\frac{T_n(f)+h_n}{(|v_k|+\frac{1}{n})^{\delta(x)}}-\frac{T_n(f)+h_n}{(|v|+\frac{1}{n})^{\delta(x)}}\right )(w_k-w)\, dx\\
    &\leq \left\lVert\frac{T_n(f)+h_n}{(|v_k|+\frac{1}{n})^{\delta(x)}}-\frac{T_n(f)+h_n}{(|v|+\frac{1}{n})^{\delta(x)}}\right\rVert_{L^2(\Omega)}||w_k-w||_{L^2(\Omega)}.
\end{align*}
Using Lemma \ref{emb} in the above estimate, we get
\begin{align}\label{MeD}
    \lVert w_k-w\rVert_{L^2(\Omega)}\leq C\left\lVert\frac{T_n(f)+h_n}{(|v_k|+\frac{1}{n})^{\delta(x)}}-\frac{T_n(f)+h_n}{(|v|+\frac{1}{n})^{\delta(x)}}\right\rVert_{L^2(\Omega)},
\end{align}
where $C>0$ is a constant independent of $k$. Now since
\begin{align*}
  \left |\frac{T_n(f)+h_n}{(|v_k|+\frac{1}{n})^{\delta(x)}}\right |^2\leq 2\left (||n^{\delta(x)}||_{L^\infty(\Omega)}^2|T_n(f)+h_n|^2\right )\in L^1(\Omega),
\end{align*}
by the Lebesgue's dominated convergence theorem, we obtain $$\lim_{k\to\infty}\frac{T_n(f)+h_n}{(|v_k|+\frac{1}{n})^{\delta(x)}}= \frac{T_n(f)+h_n}{(|v|+\frac{1}{n})^{\delta(x)}}$$ strongly in $L^2(\Omega).$ Hence, (\ref{MeD}) implies $w_k\to w$ strongly in $L^2(\Omega)$ and consequently $\Theta$ is continuous.\\
{Now, we show that $\Theta$ is compact. In this concern,} let $\{v_k\}_{k\in\mathbb{N}}$ be a bounded sequence in $L^2(\Omega)$ and $w_k:=\Theta v_k$. Incorporating $w_k$ as a test function in the weak formulation of (\ref{DP1}) with $(v,w)=(v_k, w_k)$ and applying (\ref{lkernel}), (\ref{nkernel}), one has
\begin{align}\label{RR}
    \alpha\int_\Omega |\nabla w_k|^2\, dx+\Lambda^{-1}\underbrace{\int_{\mathbb{R}^N}\int_{\mathbb{R}^N}\frac{(w_k(x)-w_k(y))^2}{|x-y|^{N+2s}}\, dxdy}_{\geq 0}&\leq \int_\Omega \left (\frac{T_n(f)+h_n}{(|v_k|+\frac{1}{n})^{\delta(x)}}+g_n\right )w_k\, dx\nonumber\\
    &\leq C\lVert T_n(f)+h_n+g_n\rVert_{L^2(\Omega)}\lVert w_k\rVert_{L^2(\Omega)}\nonumber\\
    &\leq C\lVert T_n(f)+h_n+g_n\rVert_{L^2(\Omega)}\lVert \nabla w_k\rVert_{L^2(\Omega)},
\end{align}
{where $C>0$ is a constant independent on $k.$ To obtain the last inequality we have again used Lemma \ref{emb}). The inequality (\ref{RR}) yields}
$$\lVert \nabla w_k\rVert_{L^2(\Omega)}\leq C\lVert T_n(f)+h_n+g_n\rVert_{L^2(\Omega)}, $$
where $C>0$ is a constant independent of $k.$
Thus, {since $h_n,g_n\in L^2(\Omega)$, the above estimates yields that the sequence} $\{w_k\}_{k\in\mathbb{N}}$ is bounded in $W^{1,2}_0(\Omega).$ By {Lemma \ref{emb}}, the sequence $\{w_k\}_{k\in\mathbb{N}}$ has a convergent subsequence in $L^2(\Omega).$ This proves that $\Theta$ is compact. {Using the same argument above, we observe that there exists a radius $R>0$ such that the ball of radius $R$ in $L^2(\Omega)$ is invariant under the mapping $\Theta$. Hence, one can exploit Schauder's fixed point theorem to conclude the existence of a fixed point of the map $\Theta$, i.e., there exists $u_n\in W^{1,2}_0(\Omega)$ such that
\begin{align}\label{A66}
    \begin{cases}
        \mathcal{M}u_n=\frac{T_n(f)+h_n}{(|u_n|+\frac{1}{n})^{\delta(x)}}+g_n \text{ in } \Omega,\\
        u_n=0 \text{ in }\mathbb{R}^N\setminus\Omega.
    \end{cases}
\end{align} }
We put $\phi=-u_n^-$ in the weak formulation of the above equation (\ref{A66}) and obtain $\int_\Omega|\nabla u_n^-|^2\, dx\leq 0$ that gives $u_n^-=0$ in $\Omega$ and hence $u_n\geq 0$ in $\Omega$. {By \cite[Theorem 3.10]{GKK}, we have $u_n>0$ in $\Omega$ and consequently it is a weak solution of the equation (\ref{APE1}).} \\
\textbf{Uniqueness:}
Let $n\in\mathbb{N}$ and suppose $u_n$ and $\Tilde{u}_n$ in $W_0^{1,2}(\Omega)$ are two weak solutions of the equation (\ref{APE1}). Thus, $w=u_n-\Tilde{u}_n$ is a weak solution of the equation
\begin{align}\label{APP}
    \begin{cases}
        \mathcal{M}w=\left (\frac{T_n(f)+h_n}{(u_n+\frac{1}{n})^{\delta(x)}}-\frac{T_n(f)+h_n}{(\Tilde{u}_n+\frac{1}{n})^{\delta(x)}}\right ) \text{ in } \Omega,\\
    w=0 \text{ in } \mathbb{R}^N\setminus\Omega.
    \end{cases}
\end{align}
Choosing $\phi=w^+$ as a test function in the weak formulation of (\ref{APP}) and utilizing (\ref{lkernel}), (\ref{nkernel}) we obtain
\begin{align}
    \alpha \int_\Omega|\nabla w^+|^2\, dx\leq \int_\Omega \left (\frac{T_n(f)+h_n}{(u_n+\frac{1}{n})^{\delta(x)}}-\frac{T_n(f)+h_n}{(\Tilde{u}_n+\frac{1}{n})^{\delta(x)}}\right )w^+\, dx\leq 0,
\end{align}
which reveals that $w^+=0$ in $\Omega$. Consequently, $u_n\leq \Tilde{u}_n$ in $\Omega$. Interchanging the roles of $u_n$ and $\Tilde{u}_n$, we obtain $\Tilde{u}_n\leq u_n$ in $\Omega$. This proves that the equation (\ref{APE1}) admits a unique weak solution $u_n\in W^{1,2}_0(\Omega)$. \\
\textbf{Uniform positivity:}
Let $n\in\mathbb{N}$. To obtain a uniform lower bound of $u_n$, we compare $u_n$ with the solution $w\in W^{1,2}_0(\Omega)$ of the following equation:
\begin{align}\label{A77}
    \begin{cases}
         \mathcal{M}w=\frac{T_1(f)}{(w+1)^{\delta(x)}}\text{ in }\Omega,\\
     w=0 \text{ in }\mathbb{R}^N\setminus\Omega \text{ and } w>0 \text{ in }\Omega.
    \end{cases}
\end{align}
Existence of such function $w$ follows from \cite[Lemma 4.6]{GKK}. By incorporating the test function $\psi=(w-u_n)^+$ in the weak formulations of (\ref{APE1}),(\ref{A77}) and subtracting one from the other, we obtain
\begin{align*}
    \int_\Omega A(x)&\nabla(w-u_n)\cdot\nabla\psi\, dx+\int_{\mathbb{R}^N}\int_{\mathbb{R}^N}((w-u_n)(x)-(w-u_n)(y))(\psi(x)-\psi(y))K(x,y)dx\, dy\nonumber\\
    &=\int_\Omega\frac{T_1(f)\psi}{(w+1)^{\delta(x)}}\, dx
    -\int_\Omega\frac{T_n(f)\psi}{(u_n+\frac{1}{n})^{\delta(x)}}\, dx-\int_\Omega \frac{h_n\psi}{(u_n+\frac{1}{n})^{\delta(x)}}\, dx-\int_\Omega g_n\psi \, dx\nonumber\\
    &\leq \int_\Omega \left ( \frac{1}{(w+1)^{\delta(x)}}-\frac{1}{(u_n+\frac{1}{n})^{\delta(x)}}\right )T_n(f)\psi\, dx\leq 0.
\end{align*}
Due to the property (\ref{lkernel}) and (\ref{nkernel}), we have 
\begin{align*}
    \alpha\int_\Omega|\nabla \psi|^2 dx+\underbrace{\Lambda^{-1}\int_{\mathbb{R}^N}\int_{\mathbb{R}^N}\frac{((w-u_n)(x)-(w-u_n)(y))(\psi(x)-\psi(y))}{|x-y|^{N+2s}}\, dxdy}_{\geq 0}\leq 0,
\end{align*}
which implies
\begin{align*}
    0\leq\int_{\Omega}|\nabla\psi|^2\, dx\leq 0.
\end{align*}
 Thus, $\psi=0$ in $\Omega$ and hence $u_n\geq w$ in $\Omega$. Since $n\in\mathbb{N}$ is arbitrary, this inequality holds for every $n\in\mathbb{N}$. {Thus, from \cite[Lemma 4.6]{GKK}, we can conclude that for every $\omega\Subset\Omega$, there exists a constant $C(\omega)>0$ such that $w\geq C(\omega)$ in $\omega$ and consequently, $u_n\geq C(\omega)$ in $\omega$ for all $n$.}
\end{proof}

\subsection{A priori estimates for {Theorem \ref{T3} and Theorem \ref{T1}}}
The following uniform boundeness results will be useful for the proofs of Theorem \ref{T3} and Theorem \ref{T1}.
\begin{Lemma}\label{LD2}(Uniform boundedness)
    Assume that the function $\delta:\overline{\Omega}\to (0,\infty)$ is continuous which satisfies the condition $(P_{\epsilon,\delta_*})$ for some $\delta_*\geq 1$ and {for some }$\epsilon>0$. Let $n\in\mathbb{N}$ and assume that the solution of (\ref{APE1}) obtained in Lemma \ref{LD11} is denoted by $u_n$. If the sequences of non-negative functions $\{g_n\}_{n\in\mathbb{N}},\{h_n\}_{n\in\mathbb{N}}\subset L^2(\Omega)$ in $\Omega$ are uniformly bounded in $L^1(\Omega)$, then the following conclusions hold:
    \begin{enumerate}
        \item[(a)] If $\delta_*=1$, then the sequence $\{u_n\}_{n\in \mathbb{N}}$ is uniformly bounded in $W^{1,q}_0(\Omega)$ for every $1<q<\frac{N}{N-1}$.
        \item[(b)] If $\delta_*>1$, then the sequence $\{u_n\}_{n\in \mathbb{N}}$ is uniformly bounded in $W^{1,q}_{\mathrm{loc}}(\Omega)\cap L^1(\Omega)$ for every $1<q<\frac{N}{N-1}$. {Moreover, the sequences $\{G_k(u_n)\}_{n\in\mathbb{N}}$, $\{T_k(u_n)\}_{n\in\mathbb{N}}$ and $\Big\{T_k^\frac{\delta_*+1}{2}(u_n)\Big\}_{n\in\mathbb{N}}$ are uniformly bounded in $W^{1,q}_0(\Omega)$, $W^{1,2}_{\mathrm{loc}}(\Omega)$ and $W^{1,2}_0(\Omega)$ respectively for every fixed $k>0$ and $1<q<\frac{N}{N-1}$.}
    \end{enumerate}
\end{Lemma}
\begin{proof}
\begin{enumerate}
\item[$(a)$] We assume $\delta_*=1,$ which means that there exists $\epsilon>0$ such that $0<\delta(x)\leq 1 \text{ for all }x\in\Omega_\epsilon.$ We shall prove that $\{u_n\}_{n\in\mathbb{N}}$ is uniformly bounded in $W^{1,q}_0(\Omega)$ for every $1<q<\frac{N}{N-1}$. Due to the fact (\ref{maremb}), it is sufficient to prove that $\{u_n\}_{n\in\mathbb{N}}$ is bounded in {$M^\frac{N}{N-1}(\Omega).$} Since $\Omega$ is bounded, it is enough to estimate the measure, $|x\in\Omega:\{|\nabla u_n|\geq t\}|$ for all $t\geq 1$. For any $l,t\geq 1$, we observe that
\begin{align}\label{KZ}
    |\{ x\in\Omega: |\nabla u_n|\geq t\}|\leq \underbrace{|\{x\in\Omega: |\nabla u_n|\geq t, u_n\leq l\}|}_{=I_2}+\underbrace{|\{x\in\Omega: |\nabla u_n|\geq t, u_n\geq l\}|}_{=I_1}.
\end{align}
\textbf{Estimate of $I_1$:}
For $l\geq 1$, using $\phi=T_l(u_n)$ as a test function in the weak formulation of the equation $(\ref{APE1})$ along with (\ref{lkernel}) and (\ref{nkernel}), we deduce
\begin{align}\label{AT9}
    \alpha\int_\Omega |\nabla T_l(u_n)|^2\, dx+&\Lambda^{-1}\underbrace{\int_{\mathbb{R}^N}\int_{\mathbb{R}^N} \frac{(u_n(x)-u_n(y))(T_l(u_n(x))-T_l(u_n(y)))}{|x-y|^{N+2s}}\, dxdy}_{\geq 0}\nonumber\\
    &\leq \int_\Omega\frac{T_n(f)+h_n}{(u_n+\frac{1}{n})^{\delta(x)}}T_l(u_n)\, dx+\int_\Omega T_l(u_n)g_n\, dx\nonumber\\
    &\leq \underbrace{\int_\Omega\frac{T_n(f)+h_n}{(u_n+\frac{1}{n})^{\delta(x)}}T_l(u_n)\, dx}_{=J_1}+l||g_n||_{L^1(\Omega)}.
\end{align}
We observe that, since $T_l(s)$ is an increasing function in $s$, the nonlocal integral above become non-negative.
Using the condition $(P_{\epsilon,\delta_*})$ and by Lemma \ref{LD11}, the fact $u_n\geq C$ in $\Omega\cap\Omega_\epsilon^c$ for all $n\in\mathbb{N}$, where $C>0$ is a constant independent of $n$, one has
\begin{align}\label{AT10}
J_1&=\int_{\Omega\cap\Omega_\epsilon^c}\frac{T_n(f)+h_n}{(u_n+\frac{1}{n})^{\delta(x)}}T_l(u_n)\, dx+\int_{\{x\in\Omega_\epsilon:u_n(x)\geq 1\}}\frac{T_n(f)+h_n}{(u_n+\frac{1}{n})^{\delta(x)}}T_l(u_n)\, dx\nonumber\\
    &+ \int_{\{x\in\Omega_\epsilon:u_n(x)< 1\}}\frac{T_n(f)+h_n}{(u_n+\frac{1}{n})^{\delta(x)}}T_l(u_n)\, dx\nonumber\\
    &\leq l||C^{-\delta(x)}||_{L^\infty(\Omega)} \int_{\Omega} (T_n(f)+h_n)\, dx +\int_{\{x\in\Omega_\epsilon:u_n(x)\geq 1\}}(T_n(f)+h_n)T_l(u_n)\, dx\nonumber\\
    &+\int_{\{x\in\Omega_\epsilon:u_n(x)< 1\}}(T_n(f)+h_n)\Big(u_n+\frac{1}{n}\Big)^{1-\delta(x)}\, dx\nonumber\\
    &\leq  l||C^{-\delta(x)}||_{L^\infty(\Omega)} \int_\Omega (f+h_n)\, dx+l\int_{\{x\in\Omega_\epsilon:u_n(x)\geq 1\}} (f+h_n)\, dx\nonumber\\
    &+||2^{1-\delta(x)}||_{L^\infty(\Omega)}\int_{\{x\in\Omega_\epsilon:u_n(x)< 1\}} (f+h_n)\, dx\nonumber\\
    &\leq C(l+1),
\end{align}
where $C>0$ is a constant independent of $n$.
To obtain the last inequality above, we have used the uniform $L^1$ bound of $h_n$. It follows from (\ref{AT9}) and (\ref{AT10}) along with the uniform $L^1$ bound of $g_n$ that 
\begin{align}\label{AT12}
    \int_\Omega |\nabla T_l(u_n)|^2\, dx\leq C(l+1),
\end{align}
where $C>0$ is a constant independent of $n$.
By Lemma \ref{emb}, one has
\begin{align*}
    \left (\int_{\{x\in\Omega: u_n(x)\geq l\}} |T_l{(u_n)}|^{2^*}\, dx\right )^\frac{2}{2^*}\leq C(l+1),
\end{align*}
where {$2^*=\frac{2N}{N-2}$}. Here $C>0$ is a constant independent of $n$. This yields 
\begin{align*}
    |\{x\in\Omega: u_n(x)\geq l\}|^\frac{2}{2^*}\leq \frac{C}{l}\Big(1+\frac{1}{l}\Big)\leq \frac{C}{l}.
\end{align*}
Thus, there exists a constant $C>0$ independent of $n$ such that
\begin{align}\label{ST2}
   I_1\leq \left |\{x\in\Omega: u_n(x)\geq l\}\right |\leq \frac{C}{l^{\frac{N}{N-2}}}, \text{ for all }l\geq 1.
\end{align}
\textbf{Estimate of $I_2$:} For any $t\geq 1$ and $l\geq 1$, we have
\begin{align}\label{AT14}
    I_2&=|\{x\in\Omega: |\nabla u_n|\geq t, u_n\leq l\}|\leq \frac{1}{t^2}\int_{\{x\in\Omega:u_n\leq l\}}|\nabla u_n|^2\, dx\nonumber\\
    &\leq \frac{1}{t^2}\int_\Omega |\nabla T_l(u_n)|^2\, dx\leq \frac{Cl}{t^2}\Big(1+\frac{1}{l}\Big)\leq \frac{Cl}{t^2},
\end{align}
{where $C>0$ is a constant independent of $n$. To obtain the second last inequality we have used (\ref{AT12}). Thus, combining (\ref{ST2}) and (\ref{AT14}), we conclude from (\ref{KZ}) that there exists a constant $C>0$ independent of $n$ such that}
\begin{align}\label{AT15}
    |\{x\in\Omega:|\nabla u_n|\geq t\}|\leq C\Big(\frac{1}{l^\frac{N}{N-2}}+\frac{l}{t^2}\Big), \text{ for all } l,t\geq 1.
\end{align}
    In particular, choosing $l=t^\frac{N-2}{N-1}$ in (\ref{AT15}) gives that
    \begin{align}\label{embnew}
        |\{x\in\Omega:|\nabla u_n|\geq t\}|\leq \frac{C}{t^\frac{N}{N-1}} \text { for all } t\geq 1,
    \end{align}
    for some constant $C>0$ independent of $n$.
  Thus {using the embedding \eqref{maremb} in \eqref{embnew}, it follows that} the sequence $\big\{|\nabla u_n|\big\}_{n\in\mathbb{N}}$ is uniformly bounded in $L^q(\Omega)$ for every $1< q<\frac{N}{N-1}$ and consequently, the sequence $\{u_n\}_{n\in\mathbb{N}}$ is uniformly bounded in $W^{1,q}_0(\Omega)$ for every $1<q<\frac{N}{N-1}$.
  
\item[$(b)$] We assume there exist $\delta_*>1$ and $\epsilon>0$ such that $0<\delta(x)\leq \delta_* \text{ for all } x\in\Omega_\epsilon.$\\
\textbf{Claim 1:} For every fixed $k>0$, the sequence $\{G_k(u_n)\}_{n\in\mathbb{N}}$ is uniformly bounded in $W^{1,q}_0(\Omega)$ for every $1<q<\frac{N}{N-1}.$

To this end, similar to part $(a)$ above, we estimate the measure $|\{x\in\Omega:|\nabla G_k(u_n)|\geq t\}|$ for all $t\geq 1.$ Let $l>0$. Then, we observe that 
\begin{align}\label{ineq1}
    &|\{x\in\Omega:\nabla G_k(u_n)|\geq t\}|\\
    &\leq \underbrace{|\{x\in\Omega:|\nabla G_k(u_n)|\geq t, G_k(u_n)\leq l\}}_{I_2}|+\underbrace{|\{x\in\Omega:|\nabla G_k(u_n)|\geq t, G_k(u_n)\geq l\}|}_{I_1}.
\end{align}
We estimate $I_1$ and $I_2$ using the same approach as in part $(a)$ above.\\
\textbf{ Estimate of $I_1$:}
By taking the test function $\phi=T_l(G_k(u_n))$ in the weak formulation of (\ref{APE1}), we obtain
\begin{align}\label{cl}
    \int_\Omega A(x)\nabla u_n\cdot \nabla T_l(G_k(u_n)) \, dx&+\underbrace{\int_{\mathbb{R}^N}\int_{\mathbb{R}^N}(u_n(x)-u_n(y))(\phi(x)-\phi(y)K(x,y)\, dxdy}_{\geq 0}\nonumber\\
    &=\int_\Omega \frac{(T_n(f)+h_n)}{(u_n+\frac{1}{n})^{\delta(x)}}T_l(G_k(u_n))\, dx+\int_\Omega g_nT_l(G_k(u_n))\, dx\nonumber\\
    &\leq l\left\lVert \frac{1}{k^{\delta(x)}}\right\rVert_{L^\infty(\Omega)}\int_\Omega (T_n(f)+h_n)\, dx+l\int_\Omega g_n\nonumber\\
    &\leq Cl(||f||_{L^1(\Omega)}+||h_n||_{L^1(\Omega)}+||g_n||_{L^1(\Omega)})\leq Cl,
\end{align}
where the constant $C>0$ is independent of $n.$ We observe that, since $T_l(s)$ is an increasing function in $s$, the nonlocal integral above becomes non-negative. We have used the fact that $u_n\geq k$ in $\mathrm{supp}\,T_l(G_k(u_n))$ to obtain the first inequality in \eqref{cl} above. Then the above inequality \eqref{cl} together with (\ref{lkernel}) leads to
\begin{align}\label{W4}
  \int_\Omega |\nabla T_l(G_k(u_n))|^2\, dx\leq Cl,
\end{align}
where $C>0$ is a constant independent of $n.$ Using Lemma \ref{emb}, we deduce
\begin{align*}
    |\{x\in\Omega:G_k(u_n)\geq l\}|^\frac{2}{2^*}\leq \frac{1}{l^2}\left (\int_\Omega |T_l(G_k(u_n))|^{2^*}\, dx\right )^\frac{2}{2^*}\leq \frac{1}{l^2}\int_\Omega |\nabla T_l(G_k(u_n))|^2\, dx\leq\frac{C}{l},
\end{align*}
that is,
\begin{align*}
    |\{x\in\Omega:G_k(u_n)\geq l\}|\leq \frac{C}{l^\frac{N}{N-2}},
\end{align*}
for some constant $C>0$ independent of $n$, where {$2^*=\frac{2N}{N-2}$}. Hence, 
\begin{align}\label{S4}
    I_1=|\{x\in\Omega:|\nabla G_k(u_n)|\geq t, G_k(u_n)\geq l\}|\leq \frac{C}{l^\frac{N}{N-2}},
\end{align}
for some constant $C>0$ independent of $n$.\\
\textbf{Estimate of $I_2$:}
Using the inequality (\ref{W4}), we obtain
\begin{align}\label{S5}
    I_2=|\{x\in\Omega:|\nabla G_k(u_n)|\geq t, G_k(u_n)\leq l\}|&\leq \frac{1}{t^2}\int_{\{x\in\Omega:G_k(u_n)\leq l\}}|\nabla G_k(u_n)|^2\, dx\\
    &\leq \frac{1}{t^2}\int_\Omega |\nabla T_l(G_k(u_n))|^2\, dx\nonumber\leq\frac{Cl}{t^2},
\end{align}
where $C>0$ is a constant independent of $n$.

For $t\geq 1$, choosing $l=t^\frac{N-2}{N-1}$ in (\ref{S4}) and (\ref{S5}) and using these two inequalities in (\ref{ineq1}), we get
\begin{align*}
    \left |\{x\in\Omega:|\nabla G_k(u_n)|\geq t\}\right |\leq \frac{C}{t^\frac{N}{N-1}} ,
\end{align*}
for some constant $C>0$ independent of $n$. This completes the proof of Claim 1.\\
\textbf{Claim 2:} For every fixed $k>0$, the sequences $\{T_k(u_n)\}_{n\in\mathbb{N}}$ and $\Big\{T_k(u_n)^\frac{\delta_*+1}{2}\Big\}_{n\in\mathbb{N}}$ are uniformly bounded in $W^{1,2}_{\mathrm{loc}}(\Omega)$ and $W^{1,2}_{0}(\Omega)$ respectively.

By choosing $\phi=T^{\delta_*}_k(u_n)$ into the weak formulation of (\ref{APE1}) and making use of the properties (\ref{lkernel}) and (\ref{nkernel}), we deduce
\begin{align}\label{W2}
    \frac{4\alpha\delta_*}{(\delta_*+1)^2}\int_\Omega \Big|\nabla T^\frac{\delta_*+1}{2}_k(u_n)\Big|^2\, dx&+\Lambda^{-1}\underbrace{\int_{\mathbb{R}^N}\int_{\mathbb{R}^N} \frac{(u_n(x)-u_n(y))(\phi(x)-\phi(y))}{|x-y|^{N+2s}}\, dxdy}_{\geq 0}\nonumber\\
    &\leq \int_\Omega \frac{T_n(f)+h_n}{(u_n+\frac{1}{n})^{\delta(x)}}T^{\delta_*}_k(u_n)\, dx+\int_\Omega T^{\delta_*}_k(u_n) g_n\, dx.\nonumber\\
    &\leq \int_\Omega \frac{T_n(f)+h_n}{(u_n+\frac{1}{n})^{\delta(x)}}T_k^{\delta_*}(u_n)\, dx+k^{\delta_*}\int_\Omega g_n\, dx.
\end{align}
Using the condition $(P_{\epsilon,\delta_*})$ and by Lemma \ref{LD11}, the fact $u_n\geq C$ in $\Omega\cap\Omega_\epsilon^c$ for all $n\in\mathbb{N}$, where $C>0$ is a constant independent of $n$, we obtain
\begin{align}\label{TB}
    \int_\Omega &\frac{T_n(f)+h_n}{(u_n+\frac{1}{n})^{\delta(x)}}T_k^{\delta_*}(u_n)\, dx=\int_{\Omega\cap\Omega_\epsilon^c}\frac{T_n(f)+h_n}{(u_n+\frac{1}{n})^{\delta(x)}}T_k^{\delta_*}(u_n)\, dx\nonumber\\
    &+\int_{\{x\in\Omega_\epsilon:u_n(x)\geq 1\}}\frac{T_n(f)+h_n}{(u_n+\frac{1}{n})^{\delta(x)}}T_k^{\delta_*}(u_n)\, dx+ \int_{\{x\in\Omega_\epsilon:u_n(x)< 1\}}\frac{T_n(f)+h_n}{(u_n+\frac{1}{n})^{\delta(x)}}T_k^{\delta_*}(u_n)\, dx\nonumber\\
    &\leq ||C^{-\delta(x)}||_{L^\infty(\Omega)} k^{\delta_*}\int_{\Omega} (T_n(f)+h_n)\, dx+k^{\delta_*}\int_{\{x\in \Omega_\epsilon:u_n(x)\geq 1\}}(T_n(f)+h_n)\, dx\nonumber\\
    &+\int_{\{x\in\Omega_\epsilon:u_n(x)< 1\}}(T_n(f)+h_n)\Big(u_n+\frac{1}{n}\Big)^{\delta_*-\delta(x)}\, dx\nonumber\\
    &\leq  ||C^{-\delta(x)}||_{L^\infty(\Omega)}k^{\delta_*} \int_\Omega (f+h_n)\, dx+k^{\delta_*}\int_{\{x\in\Omega_\epsilon:u_n(x)\geq 1\}} (f+h_n)\, dx\nonumber\\
    &+||2^{\delta_*-\delta(x)}||_{L^\infty(\Omega)}\int_{\{x\in\Omega_\epsilon:u_n(x)< 1\}} (f+h_n)\, dx\nonumber\\
    &\leq C(k^{\delta_*}+1),
\end{align}
where $C>0$ is a constant independent of $n$. To obtain the last inequality above, we have used the facts $f\in L^1(\Omega)$ and that $\{ h_n\}_{n\in\mathbb{N}}$ is uniformly bounded in ${L^1(\Omega)}.$ Using the estimate \eqref{TB} and uniform $L^1(\Omega)$ bound of $g_n$, (\ref{W2}) yields
\begin{align}\label{clm2}
    \int_\Omega\Big|\nabla T^\frac{\delta_*+1}{2}_k(u_n)\Big|^2\, dx\leq C(k^{\delta_*}+1),
\end{align}
where $C>0$ is a constant independent of $n$. Hence, the sequence $\Big\{T_k(u_n)^\frac{\delta_*+1}{2}\Big\}_{n\in\mathbb{N}}$ is uniformly bounded in $W^{1,2}_{0}(\Omega)$ for every fixed $k>0$. Moreover, since for every $\omega\Subset\Omega$, there exists a constant $C(\omega)>0$ (independent of $n$) such that $u_n\geq C(\omega)$ for all $n$, therefore using \eqref{clm2}, we obtain
\begin{align*}
    C(\omega)^{\delta_*-1}\int_\omega |\nabla T_k(u_n)|^2\, dx\leq {\int_\omega u_n^{\delta_*-1}|\nabla T_k(u_n)|^2\, dx=\int_\Omega T^{\delta_*-1}_k(u_n)|\nabla T_k(u_n)|^2\, dx}\leq C,
\end{align*}
where $C>0$ is a constant independent of $n$.
Hence, $\{T_k(u_n)\}_{n\in\mathbb{N}}$ is uniformly bounded in $W^{1,2}_{\mathrm{loc}}(\Omega).$ This completes the proof of Claim 2.

Finally, since $u_n=T_1(u_n)+G_1(u_n)$, using Claim 1 and Claim 2, it follows that the sequence $\{u_n\}_{n\in\mathbb{N}}$ is uniformly bounded in $W^{1,q}_{\mathrm{loc}}(\Omega)\cap L^1(\Omega)$ for every $1<q<\frac{N}{N-1}$.
\end{enumerate}
\end{proof}

\subsection{A priori estimates for {Theorem \ref{T2}}}
\begin{Lemma}\label{BL2}
Assume that the function $\delta:\overline{\Omega}\to (0,\infty)$ is a constant function. Let $n\in\mathbb{N}$ and define $g_n(x)=T_n(g(x))=\min\{g(x),n\}$ {in \eqref{APE1}}, where $g\in L^{\frac{N(\delta+1)}{N+2\delta}}(\Omega)$ is a non-negative function in $\Omega$. {Suppose that the solution of \eqref{APE1} obtained in Lemma \ref{LD11} is denoted by $u_n$.} Assume that the sequence of non-negative functions $\{h_n\}_{n\in\mathbb{N}}\subset L^2(\Omega)$ in $\Omega$ is uniformly bounded in $L^1(\Omega)$. Then the following conclusions hold:
\begin{enumerate}
    \item[$(a)$] If $\delta=1$, then $\{u_n\}_{n\in\mathbb{N}}$ is uniformly bounded in $W^{1,2}_0(\Omega)$.
    \item[$(b)$] If $0<\delta<1$, then $\{u_n\}_{n\in\mathbb{N}}$ is uniformly bounded in $W^{1,q}_0(\Omega)$, where $q=\frac{N(\delta+1)}{N+\delta-1}$.
     \item[$(c)$] If $\delta>1$, then $\{u_n\}_{n\in\mathbb{N}}$ is uniformly bounded in $W^{1,2}_{\mathrm{loc}}(\Omega)$. Moreover, $\Big\{u^\frac{\delta+1}{2}_n\Big\}_{n\in\mathbb{N}}$ is uniformly bounded in $W^{1,2}_0(\Omega).$
\end{enumerate}
\end{Lemma}

\begin{proof}
\begin{enumerate}
\item[$(a)$] Choosing $u_n$ as a test function in the weak formulation of (\ref{APE1}) and applying the properties (\ref{lkernel}), (\ref{nkernel}), we get
    \begin{align}\label{RR2}
        \alpha\int_\Omega |\nabla u_n|^2\, dx+&\Lambda^{-1}\underbrace{\int_{\mathbb{R}^N}\int_{\mathbb{R}^N}\frac{(u_n(x)-u_n(y))^2}{|x-y|^{N+2s}}\, dxdy}_{\geq 0}\nonumber\\
        &\leq\int_\Omega\frac{(T_n(f)+h_n)u_n}{(u_n+\frac{1}{n})}\, dx+\int_\Omega g_nu_n\, dx.
    \end{align}
    Using H\"{o}lder's inequality along with Lemma \ref{emb} in (\ref{RR2}), it follows that 
$$\alpha||u_n||_{{W_0^{1,2}(\Omega)}}^2\leq \int_\Omega(f+{h_n})\, dx +\int gu_n\, dx\leq C+||g||_{L^{(2^*)'}(\Omega)}||u_n||_{L^{2^*}(\Omega)}\leq C(1+||u_n||_{{W_0^{1,2}(\Omega)}}),$$
where $C$ is a constant independent of $n$. Hence the sequence $\{u_n\}_{n\in\mathbb{N}}$ is uniformly bounded in $W^{1,2}_0(\Omega).$

\item[$(b)$] If $0<\delta<1$, then for $0<\epsilon<\frac{1}{n}$, choosing $\phi=(u_n+\epsilon)^\delta-\epsilon^\delta$ as a test function in the weak formulation of (\ref{APE1}) and using the properties (\ref{lkernel}), (\ref{nkernel}), we obtain
\begin{align}\label{RR1}
     \alpha\int_\Omega\Big|\nabla (u_n+\epsilon)^{\frac{\delta+1}{2}}\Big|^2\, dx+&\Lambda^{-1}\underbrace{\int_{\mathbb{R}^N}\int_{\mathbb{R}^N}\frac{(u_n(x)-u_n(y))(\phi(x)-\phi(y))}{|x-y|^{N+2s}}\, dxdy}_{\geq 0}\\
     &\leq \int_\Omega\frac{(T_n(f)+h_n)\phi}{(u_n+\frac{1}{n})^\delta}\, dx+\int_\Omega g_n{\phi}\, dx\\
     &\leq \int_\Omega (T_n(f)+h_n)\, dx+\int_\Omega g(u_n+\epsilon)^\delta\, dx\nonumber\\
     &\leq C+\int_\Omega g(u_n+\epsilon)^\delta\, dx,
\end{align}
where $C$ is a constant independent of $n$. Using H\"{o}lder inequality, (\ref{RR1}) yields
\begin{align}\label{BW}
    \int_\Omega\Big|\nabla (u_n+\epsilon)^{\frac{\delta+1}{2}}\Big|^2\, dx\leq C\left [1+\left (\int_\Omega (u_n+\epsilon)^\frac{2^*(\delta+1)}{2}\, dx\right )^\frac{2\delta}{2^*(\delta+1)}||g||_{L^r(\Omega)}\right ],
\end{align}
where $r=\frac{N(\delta+1)}{N+2\delta}$ {and $C>0$ is a constant independent of $n$.} {Here $2^*=\frac{2N}{N-2}$.}
Applying Lemma \ref{emb} in \eqref{BW}, we have
\begin{align}
    \left (\int_\Omega|u_n+\epsilon|^{\frac{2^*(\delta+1)}{2}}\, dx\right )^\frac{2}{2^*}\leq C\left [1+\left (\int_\Omega (u_n+\epsilon)^\frac{2^*(\delta+1)}{2}\, dx\right )^\frac{2\delta}{2^*(\delta+1)}||g||_{L^r(\Omega)}\right ],
\end{align}
where $C$ is a constant independent of $n$.
Hence, there exists a constant $C>0$ independent of $n$ such that 
\begin{align}\label{BM1}
\int_\Omega|u_n+\epsilon|^{\frac{2^*(\delta+1)}{2}}\, dx\leq C.
\end{align}
Using \eqref{BM1} in (\ref{BW}), we arrive at
\begin{align}
    \int_\Omega \frac{|\nabla u_n|^2}{(u_n+\epsilon)^{(1-\delta)}}\, dx\leq C,
\end{align}
where $C$ is a constant independent of $n$.
Now, for any {$1<p<2$}, using H\"{o}lder's inequality, we have
\begin{align}\label{BM2}
    \int_\Omega|\nabla u_n|^p\, dx&\leq \int_\Omega\frac{|\nabla u_n|^p}{(u_n+\epsilon)^\frac{(1-\delta)p}{2}}(u_n+\epsilon)^\frac{(1-\delta)p}{2}\,dx\nonumber\\
    &\leq \left (\int_\Omega \frac{|\nabla u_n|^2}{(u_n+\epsilon)^{(1-\delta)}}\,dx\right )^\frac{p}{2}\left (\int_\Omega (u_n+\epsilon)^\frac{(1-\delta)p}{2-p}\,dx\right )^\frac{2-p}{2}\nonumber\\
    &\leq C\left (\int_\Omega (u_n+\epsilon)^\frac{(1-\delta)p}{2-p}\,dx\right )^\frac{2-p}{2},
\end{align}
where $C$ is a constant independent of $n$. We choose {$1<p<2$} such that $\frac{(1-\delta)p}{2-p}=\frac{2^*(\delta+1)}{2}$, i.e., $p=\frac{N(\delta+1)}{N+\delta-1}=q.$ Note that $1<q<2$, since $N>2$ and $\delta\in(0,1)$. Hence, the proof follows from (\ref{BM1}) and (\ref{BM2}).

\item[$(c)$] {Let $k>0$, then we choose $T_k^\delta(u_n)$ as a test function in the weak formulation of the equation \eqref{APE1} with $\delta>1$ and apply (\ref{lkernel}), (\ref{nkernel}) to obtain
\begin{align*}
\alpha\delta\int_\Omega\Big|\nabla T_k^\frac{\delta+1}{2}(u_n)\Big|^2\, dx+&\Lambda^{-1}\underbrace{\int_{\mathbb{R}^N}\int_{\mathbb{R}^N}\frac{(u_n(x)-u_n(y))(T_k^\delta(u_n(x))-T_k^\delta(u_n(y)))}{|x-y|^{N+2s}}\, dxdy}_{\geq 0}\\
    &\leq \int_\Omega \frac{T_k^\delta(u_n)(T_n(f)+h_n)}{(u_n+\frac{1}{n})^\delta}\, dx\nonumber+\int_\Omega g_nT_k^\delta(u_n)\, dx.
\end{align*}
Using the fact $T_k(s)\leq s$ for every $k,s>0$, the above identity yields
\begin{align}
    \int_\Omega\Big|\nabla T_k^\frac{\delta+1}{2}(u_n)\Big|^2\, dx\leq \int_\Omega (T_n(f)+h_n)\, dx+\int_\Omega gT_k^\delta(u_n)\, dx.
\end{align}
Using H\"{o}lder's inequality and the fact that $||T_n(f)+h_n||_{L^1(\Omega)}$ is uniformly bounded, we have
\begin{align}\label{BM3}
\int_\Omega \Big|\nabla T_k^\frac{\delta+1}{2}(u_n)\Big|^2\, dx\leq C+||g||_{L^r(\Omega)}\left (\int_\Omega |T_k(u_n)|^s\, dx\right )^\frac{\delta}{s},
\end{align}
where $s=\frac{2^*(\delta+1)}{2}$, $r=\frac{N(\delta+1)}{N+2\delta}$ { and $C>0$ is a constant independent of $n,k.$} {Here $2^*=\frac{2N}{N-2}$.} Applying Lemma \ref{emb} in \eqref{BM3}, we obtain
\begin{align}\label{AB}
    \left (\int_\Omega |T_k(u_n)|^s\, dx\right )^\frac{2}{2^*}\leq C\left (1+||g||_{L^r(\Omega)}\left (\int_\Omega |T_k(u_n)|^s\, dx\right )^\frac{\delta}{s}\right ).
\end{align}
Since $\frac{2}{2^*}>\frac{\delta}{s}=\frac{2\delta}{2^*(\delta+1)}$, there exists a constant $C>0$ (independent of $n,k$) such $$\int_\Omega |T_k(u_n)|^s\, dx\leq C\text{ for all }n,k\in\mathbb{N}.$$
Thus by using (\ref{BM3}) and Fatou's lemma, we have
$$\int_\Omega \Big|\nabla u^\frac{\delta+1}{2}_n\Big|^2\, dx\leq \liminf_{k\to\infty}\int_\Omega \Big|\nabla T_k^\frac{\delta+1}{2}(u_n)\Big|^2\, dx\leq C\left (1+||g||_{L^r(\Omega)}\right ),$$
where $C>0$ is a constant independent of $n.$ Hence, the sequence $\Big\{u_n^\frac{\delta+1}{2}\Big\}_{n\in\mathbb{N}}$ is uniformly bounded in $W^{1,2}_0(\Omega).$ 
Now, for every $\omega\Subset\Omega$, there exists $C(\omega)>0$ (independent of $n$) such that $u_n\geq C(\omega)$ for all $n\in\mathbb{N}$. Thus,
\begin{align*}
    \int_\omega |\nabla u_n|^2\, dx\leq C\int_\omega \Big|\nabla u_n^\frac{\delta+1}{2}\Big|^2 u_n^{1-\delta}\, dx \leq \frac{1}{C(\omega)^{\delta-1}}\int_\Omega  \Big|\nabla u_n^\frac{\delta+1}{2}\Big|^2\, dx\leq \Tilde{C}(\omega).
\end{align*}
This completes the proof of the Lemma.}
\end{enumerate}
\end{proof}

\subsection{A priori estimates for regularity results}
In this {subsection}, unless otherwise mentioned, we assume that $\delta:\overline{\Omega}\to(0,\infty)$ is a continuous function and $f\in L^r(\Omega)\setminus\{0\}$ and $g\in L^m(\Omega)$ for some $r,m\geq 1$ are two non-negative functions in $\Omega$.

By Lemma \ref{LD11} and {proceeding along the lines of the proof of {\cite[Lemma 3.1]{Guna}}}, for each fixed $n\in\mathbb{N}$, we have the existence of a unique weak solution $u_n\in W^{1,2}_0(\Omega)\cap L^\infty(\Omega)$ of the equation
\begin{equation}\label{M00}
   \begin{cases}
        \mathcal{M} u_n=\frac{T_n(f)}{(u_n+\frac{1}{n})^{\delta(x)}}+T_n(g) \text{ in }\Omega,\\
    u_n>0 \text{ in } \Omega,\text{ and } u_n=0 \text{ in }\mathbb{R}^N\setminus\Omega.
   \end{cases}
\end{equation}
Moreover, for each fixed $n\in\mathbb{N}$, by \cite[Lemma 3.2]{Guna} there exists a unique weak solution $v_n\in W^{1,2}_0(\Omega)\cap L^\infty(\Omega)$ of the problem
\begin{align}\label{M0}
   \begin{cases}
        \mathcal{M} v_n=\frac{T_n(f)}{(v_n+\frac{1}{n})^{\delta(x)}} \text{ in }\Omega,\\
    v_n>0 \text{ in } \Omega,\text{ and } v_n=0 \text{ in }\mathbb{R}^N\setminus\Omega,
   \end{cases}
\end{align}
 and by \cite[Lemma 3.1]{Guna}, there exists a unique weak solution $w_n\in W^{1,2}_0(\Omega)\cap L^\infty(\Omega)$ solving  
\begin{align}\label{M1}
   \begin{cases}
        \mathcal{M} w_n=T_n(g) \text{ in }\Omega,\\
    w_n>0 \text{ in } \Omega,\text{ and }w_n=0 \text{ in }\mathbb{R}^N\setminus\Omega
   \end{cases}
\end{align}
respectively.
\begin{Lemma}\label{EL1}
   Let $n\in\mathbb{N}$ and assume that $u_n,\, v_n,\, w_n\in W^{1,2}_0(\Omega)$ are solutions of the problem \eqref{M00}, \eqref{M0} and \eqref{M1} respectively. {If $\{w_n\}_{n\in\mathbb{N}}$ and $\{v_n\}_{n\in\mathbb{N}}$ are uniformly bounded in $L^a(\Omega)$ and $L^b(\Omega)$ respectively for some $a,b\geq 1$, then $\{u_n\}_{n\in\mathbb{N}}$ is uniformly bounded in $L^c(\Omega)$, where $c=\min\{a,b\}$.}
\end{Lemma}
\begin{proof}
The function $\Tilde{u}_n:=(u_n-v_n-w_n)$ satisfies the following equation
\begin{align}\label{E1}
     \begin{cases}
        \mathcal{M} \Tilde{u}_n=T_n(f)\left (\frac{1}{(u_n+\frac{1}{n})^{\delta(x)}}-\frac{1}{(v_n+\frac{1}{n})^{\delta(x)}}\right ) \text{ in }\Omega,\\
    \Tilde{u}_n=0 \text{ in }\mathbb{R}^N\setminus\Omega.
   \end{cases}
\end{align}
By putting the test function $\phi=\Tilde{u}_n^+$ in the weak formulation of (\ref{E1}) and applying (\ref{lkernel}), (\ref{nkernel}), we obtain
\begin{align*}
    &\alpha\int_\Omega|\nabla \Tilde{u}_n^+|^2\, dx+\Lambda^{-1}\underbrace{\int_{\mathbb{R}^N}\int_{\mathbb{R}^N}\frac{(\Tilde{u}_n(x)-\Tilde{u_n}(y))(\Tilde{u}_n^+(x)-\Tilde{u}_n^+(y))}{|x-y|^{N+2s}}}_{\geq 0}\,dx dy\\
    &\leq \int_\Omega T_n(f)\left (\frac{1}{(u_n+\frac{1}{n})^{\delta(x)}}-\frac{1}{(v_n+\frac{1}{n})^{\delta(x)}}\right )\Tilde{u}_n^{+}\,dx\leq 0,
\end{align*}
which gives
$$\int_\Omega|\nabla \Tilde{u}_n^+|^2\, dx\leq 0,$$
that is $\Tilde{u}_n^+=0$ in $\Omega$. Therefore, $u_n\leq v_n+w_n$ in $\Omega$ and hence $\{u_n\}_{n\in\mathbb{N}}$ is uniformly bounded in $L^c(\Omega)$.
\end{proof}

The following regularity result for the above solution $w_n$ of the equation (\ref{M1}) is highly significant for our argument.  
\begin{Lemma}\label{EL2}
Assume that $g\in L^m(\Omega)$ be a non-negative function in $\Omega$ for some $m>1$. Then the above solution $w_n$ to the equation (\ref{M1})
satisfies the following conclusions:
\begin{enumerate}
    \item[(a)] If $m>\frac{N}{2}$, then $\{w_n\}_{n\in\mathbb{N}}$ is uniformly bounded in  $L^\infty(\Omega).$
    \item[(b)]  If $1<m<\frac{N}{2}$, then $\{w_n\}_{n\in\mathbb{N}}$ is uniformly bounded in $L^{m^{**}}(\Omega),$ where $m^{**}=\frac{Nm}{N-2m}$.
\end{enumerate}
\end{Lemma}
\begin{proof}
\begin{enumerate}
\item[$(a)$] For $k\geq 1$, taking $\phi=(w_n-k)^+$ as a test function in weak formulation of the equation $(\ref{M1})$ and applying (\ref{lkernel}), (\ref{nkernel}), we deduce
\begin{align*}
\alpha\int_\Omega|\nabla\phi|^2\, dx+\Lambda^{-1}\underbrace{\int_{\mathbb{R}^N}\int_{\mathbb{R}^N}\frac{(w_n(x)-w_n(y))(\phi(x)-\phi(y))}{|x-y|^{N+2s}}\, dxdy}_{\geq 0}\leq \int_\Omega T_n(g)\phi\, dx,
\end{align*}
which implies
\begin{align}\label{ab}
    \int_\Omega |\nabla \phi|^2\, dx\leq C\int_\Omega g\phi\, dx,
\end{align}
for some constant $C>0$ independent of $n$. Using Lemma \ref{emb} and the generalized H\"{o}lder's inequality in the above inequality \eqref{ab}, we get
\begin{align*}
    C\left (\int_\Omega |\phi|^{2^*}\, dx\right )^\frac{2}{2^*}\leq \int_\Omega |\nabla\phi|^2\, dx\leq C\int_\Omega g|\phi|\, dx\\
    \leq C||g||_{L^m(\Omega)}||\phi||_{L^{2^*}(\Omega)}|\{x\in\Omega:w_n\geq k\}|^{1-\frac{1}{m}-\frac{1}{2^*}},
\end{align*}
which yields
\begin{align}
    \int_\Omega|\phi|^{2^*}\, dx\leq C||g||_{L^m(\Omega)}^{2^*}|\{x\in\Omega:w_n\geq k\}|^{2^*(1-\frac{1}{m}-\frac{1}{2^*})},
\end{align}
for some constant $C>0$ independent of $n$ {and $2^*=\frac{2N}{N-2}$}.
Define $S(k):=\{x\in\Omega: w_n(x)\geq k\}$ for all $k\geq 1.$ Now, for $1\leq k\leq h$, we observe that
\begin{align*}
    |S(h)|(h-k)^{2^*}=\int_{S(h)}|h-k|^{2^*}\,dx&\leq \int_{S(k)}|(w_n-k)^+|^{2^*}\,dx=\int_\Omega |\phi|^{2^*}\,dx\nonumber\\
    &\leq C||g||_{L^m(\Omega)}^{2^*}|\{x\in\Omega:w_n\geq k\}|^{2^*(1-\frac{1}{m}-\frac{1}{2^*})}\nonumber\\
    &=C||g||_{L^m(\Omega)}^{2^*}|S(k)|^{2^*(1-\frac{1}{m}-\frac{1}{2^*})},
\end{align*}
that is, 
\begin{align*}
    |S(h)|\leq \frac{C||g||_{L^m(\Omega)}^{2^*}}{(h-k)^{2^*}}|S(k)|^\alpha,
\end{align*}
where $\alpha=2^*(1-\frac{1}{m}-\frac{1}{2^*})>1$ as $m>\frac{N}{2}.$ Thus, by \cite[ Lemma B.1]{SK}, there exists $C>0$ (independent of $n$) such that $||w_n||_{L^\infty(\Omega)}\leq C$ in $\Omega$, for all $n\in\mathbb{N}$.

\item[$(b)$] For $\epsilon>0$, treating $\phi=(w_n+\epsilon)^\gamma-\epsilon^\gamma$ ($\gamma>0$ to be determined later) as a test function in the weak formulation of $(\ref{M1})$ and using (\ref{lkernel}), (\ref{nkernel}), we obtain
\begin{align*}
 \alpha \int_\Omega\Big|\nabla (w_n+\epsilon)^\frac{\gamma+1}{2}\Big|^2\, dx+\Lambda^{-1}\underbrace{\int_{\mathbb{R}^N}\int_{\mathbb{R}^N}\frac{(w_n(x)-w_n(y))(\phi(x)-\phi(y))}{|x-y|^{N+2s}}\, dxdy}_{\geq 0}\\\leq \int_\Omega T_n(g)\phi\, dx,
\end{align*}
which implies
\begin{align*}
    \int_\Omega\Big|\nabla (w_n+\epsilon)^\frac{\gamma+1}{2}\Big|^2\,dx\leq C \int_\Omega g(w_n+\epsilon)^{\gamma}\, dx\leq C||g||_{L^m(\Omega)}\left (\int_\Omega (w_n+\epsilon)^{m'\gamma}\,dx\right )^\frac{1}{m'},
\end{align*}
{for some constant $C>0$ independent of $n$.} Using Lemma \ref{emb}, we get
\begin{align}
    \left (\int_\Omega (w_n+\epsilon)^\frac{2^*(\gamma+1)}{2}\, dx\right )^\frac{2}{2^*}\leq C||g||_{L^m(\Omega)}\left (\int_\Omega (w_n+\epsilon)^{m'\gamma}\, dx\right )^\frac{1}{m'},
\end{align}
where $C>0$ is a constant independent of $n$ and {$2^*=\frac{2N}{N-2}$.} We choose $\gamma>0$ such that $\frac{2^*(\gamma+1)}{2}=m'\gamma$, i.e. $\gamma=\frac{N(m-1)}{(N-2m)}$. Therefore, $m'\gamma=\frac{Nm}{N-2m}=m^{**}$. Since $1<m<\frac{N}{2}$, one has
\begin{align*}
     \int_\Omega (w_n+\epsilon)^{m^{**}}\, dx\leq C,
\end{align*}
where $C>0$ is a constant independent of $n$.
Finally, by Fatou's Lemma, the result follows.
\end{enumerate}
\end{proof}
\begin{Lemma}\label{EL5}
    Suppose $\delta:\overline{\Omega}\to(0,\infty)$ is a continuous function satisfying $(P_{\epsilon,\delta_*})$ for some $\delta_*\geq 1$ and for some $\epsilon>0$. Let $f\in L^r(\Omega)\setminus\{0\}$ be a non-negative function in $\Omega$ for some {$r\geq 1$}. For each $n\in\mathbb{N}$, let $v_n$ be the unique weak solution of the problem (\ref{M0}). Then the following conclusions hold:
\begin{enumerate}
    \item[(a)] If $r>\frac{N}{2}$, then the sequence $\{v_n\}_{n\in\mathbb{N}}$ is uniformly bounded in $L^\infty(\Omega).$
    \item[(b)] If ${\frac{N(\delta_*+1)}{N+2\delta_*}}\leq r<\frac{N}{2}$, then the sequence $\{v_n\}_{n\in\mathbb{N}}$ is uniformly bounded in $L^{r^{**}}(\Omega),$
    where $r^{**}=\frac{Nr}{N-2r}$.
\end{enumerate}
\end{Lemma}
\begin{proof}
\begin{enumerate}
\item[$(a)$] For $k\geq 1$, incorporating the test function $\phi=(v_n-k)^+$ in the weak formulation of the equation (\ref{M0}) and using (\ref{lkernel}), (\ref{nkernel}), we obtain
\begin{align}\label{M6}
\alpha\int_\Omega |\nabla \phi|^2\, dx+\Lambda^{-1}\underbrace{\int_{\mathbb{R}^N}\int_{\mathbb{R}^N}\frac{(v_n(x)-v_n(y))(\phi(x)-\phi(y))}{|x-y|^{N+2s}}\, dxdy}_{\geq 0}\leq\int_\Omega \frac{\phi T_n(f)}{(v_n+\frac{1}{n})^{\delta(x)}}\,dx.
\end{align}
Since $v_n\geq k\geq 1$ on $\mathrm{supp}\,\phi$, we have $\big|\frac{\phi}{(v_n+\frac{1}{n})^{\delta(x)}}\big|\leq \phi$ {on $\mathrm{supp}\,\phi$}. Thus, (\ref{M6}) yields 
\begin{align}
    \int_\Omega |\nabla\phi|^2\,dx\leq C\int_\Omega \phi f\,dx,
\end{align}
{for some constant $C>0$ independent of $n$.}
    Rest of the proof follows in a similar way to the Lemma \ref{EL2}-$(a)$.
    
\item[$(b)$] Since $v_n\in L^\infty(\Omega)$, we choose the test function $\phi=v_n^\gamma$ ($\gamma\geq \delta_*$ to be determined later) in the weak formulation of (\ref{M0}) and apply the properties (\ref{lkernel}), (\ref{nkernel}) to get
\begin{align}\label{M7}
 \alpha\int_\Omega \Big|\nabla v_n^\frac{\gamma+1}{2}\Big|^2\, dx+\Lambda^{-1}\underbrace{\int_{\mathbb{R}^N}\int_{\mathbb{R}^N} \frac{(v_n(x)-v_n(y))(\phi(x)-\phi(y))}{|x-y|^{N+2s}}\, dxdy}_{\geq 0}\leq \int_\Omega \frac{\phi T_n(f)}{(v_n+\frac{1}{n})^{\delta(x)}}\, dx.  
\end{align}
Now, using the condition $(P_{\epsilon,\delta_*})$, we get
\begin{align}\label{M8}
\int_\Omega \frac{\phi T_n(f)}{(v_n+\frac{1}{n})^{\delta(x)}}\, dx&\leq \int_{\Omega\cap\Omega_\epsilon^c}\frac{v_n^\gamma f}{(v_n+\frac{1}{n})^{\delta(x)}}\, dx+\int_{\{x\in\Omega_\epsilon: v_n\leq 1\}} \frac{v_n^\gamma f}{(v_n+\frac{1}{n})^{\delta(x)}}\, dx\nonumber\\
&+\int_{\{x\in\Omega_\epsilon: v_n>1\}} \frac{v_n^\gamma f}{(v_n+\frac{1}{n})^{\delta(x)}}\, dx\nonumber \\
&\leq \Big\|\frac{1}{C(\epsilon)^{\delta(x)}}\Big\|_{L^\infty(\Omega)}\int_\Omega v_n^\gamma f\, dx+\int_{\{x\in\Omega_\epsilon: v_n\leq 1\}} v_n^{\gamma-\delta(x)} f\, dx\nonumber\\
&+\int_{\{x\in\Omega_\epsilon: v_n>1\}} v_n^\gamma f\, dx\nonumber\\
&\leq C\left (\int_\Omega v_n^\gamma f\, dx+\int_{\{x\in\Omega_\epsilon: v_n\leq 1\}} f\, dx+\int_{\{x\in\Omega_\epsilon: v_n>1\}} v_n^\gamma f\, dx\right )\nonumber\\
&\leq C\left (\int_\Omega v_n^\gamma f\, dx+||f||_{L^r(\Omega)}\right )\nonumber\\
&\leq C||f||_{L^r(\Omega)}\left [ 1+\left (\int_\Omega |v_n|^{r'\gamma}\, dx\right )^\frac{1}{r'} \right],
\end{align}
for some constant $C>0$ independent of $n$. Combining (\ref{M7}) and (\ref{M8}) we deduce
 \begin{align*}
     \int_\Omega\Big|\nabla v_n^\frac{\gamma+1}{2}\Big|^2\, dx\leq C||f||_{L^r(\Omega)}\left [ 1+\left (\int_\Omega |v_n|^{r'\gamma}\, dx\right )^\frac{1}{r'} \right],
 \end{align*}
 for some constant $C>0$ independent of $n$. By Lemma \ref{emb}, we obtain
   \begin{align}\label{M9}
    \left (\int_\Omega |v_n|^\frac{2^*(\gamma+1)}{2}\, dx\right )^\frac{2}{2^*}\leq C||f||_{L^r(\Omega)}\left [ 1+\left (\int_\Omega |v_n|^{r'\gamma}\, dx\right )^\frac{1}{r'} \right],
\end{align}
for some constant $C>0$ independent of $n$. {Here $2^*=\frac{2N}{N-2}$.} We choose $\gamma$ such that $\frac{2^*(\gamma+1)}{2}=r'\gamma$, i.e., $\gamma=\frac{N(r-1)}{(N-2r)}$. Since $\frac{N(\delta_*+1)}{N+2\delta_*}\leq r<\frac{N}{2}$, we have $\gamma\ge\delta_*$ and $\frac{1}{r'}<\frac{2}{2^*}$. Thus inequality (\ref{M9}) gives that $\{v_n\}_{n\in\mathbb{N}}$ is uniformly bounded in $L^{r^{**}}(\Omega)$, where $r^{**}=r'\gamma=\frac{Nr}{N-2r}.$
    \end{enumerate}
\end{proof}
For Lemma \ref{EL4} and Lemma \ref{EL3} below, we assume that $\delta:\overline{\Omega}\to(0,\infty)$ is a constant function.
\begin{Lemma}\label{EL4}
    Let $0<\delta<1$ and suppose that $f\in L^r(\Omega)\setminus\{0\}$ is a non-negative function in $\Omega$ for some $r\geq 1$. For each fixed $n\in\mathbb{N}$, let $v_n$ be the unique weak solution of the problem (\ref{M0}). Then the following conclusions hold:
\begin{enumerate}
    \item[$(a)$] If $r>\frac{N}{2}$, then the sequence $\{v_n\}_{n\in\mathbb{N}}$ is uniformly bounded in $L^\infty(\Omega).$
    \item[$(b)$] If $\left (\frac{2^*}{1-\delta}\right )'\leq r<\frac{N}{2}$, then the sequence $\{v_n\}_{n\in\mathbb{N}}$ is uniformly bounded in $L^s(\Omega),$
    where $s=\frac{Nr(\delta+1)}{N-2r}$.
\end{enumerate}
\end{Lemma}
\begin{proof}
\begin{enumerate}
\item[$(a)$] The proof follows from Lemma \ref{EL5}.

\item[$(b)$] To begin with, we consider the case $r=\left (\frac{2^*}{1-\delta}\right )'$. By incorporating $v_n$ as a test function in the weak formulation in (\ref{M0}) and applying (\ref{lkernel}), (\ref{nkernel}) to deduce
$$\alpha\int_\Omega |\nabla v_n|^2\, dx\leq \int_\Omega fv_n^{1-\delta}\, dx\leq ||f||_{L^r(\Omega)}||v_n||_{L^{2^*}(\Omega)}^{1-\delta}\leq ||f||_{L^r(\Omega)}||v_n||_{{W_0^{1,2}(\Omega)}}^{1-\delta}, $$
which yields $\{v_n\}_{n\in\mathbb{N}}$ is uniformly bounded in $W^{1,2}_0(\Omega).$
Thus,  by Lemma \ref{emb}, $\{v_n\}_{n\in\mathbb{N}}$ is uniformly bounded in $L^{2^*}(\Omega)$. This proves the conclusion for $r=\left (\frac{2^*}{1-\delta}\right )'$. For the case $\left (\frac{2^*}{1-\delta}\right )'< r<\frac{N}{2}$, we put $v_n^{2\gamma-1}$ ($\gamma>1$ to be determined later) as a test function in the weak formulation of the equation (\ref{M0}) and obtain
\begin{align*}
\int_\Omega |\nabla v_n^\gamma|^2\, dx\leq C\int_\Omega v_n^{2\gamma-\delta-1}f\, dx\leq C||f||_{L^r(\Omega)}\left (\int_\Omega v_n^{r'(2\gamma-\delta-1)}\, dx\right )^\frac{1}{r'},
\end{align*}
for some constant $C>0$ independent of $n$.
By Lemma \ref{emb}, we infer that
\begin{align}\label{M5}
     \left (\int_\Omega v_n^{2^*\gamma}\, dx\right )^\frac{2}{2^*}\leq C||f||_{L^r(\Omega)}\left (\int_\Omega v_n^{r'(2\gamma-\delta-1)}\, dx\right )^\frac{1}{r'},
\end{align}
{for some constant $C>0$ independent of $n$.} {Here $2^*=\frac{2N}{N-2}$.}
 We choose $\gamma$ in such a way that $2^*\gamma=r'(2\gamma-\delta-1)$, i.e., $\gamma=\frac{r(\delta+1)(N-2)}{2(N-2r)}>1$ \big(since $r>\left (\frac{2^*}{1-\delta}\right)'$\big). From the fact $r<\frac{N}{2}$ and (\ref{M5}) we can conclude that $\{v_n\}_{n\in\mathbb{N}}$ is bounded in $L^s(\Omega)$ with $s=\frac{Nr(\delta+1)}{(N-2r)}.$
 \end{enumerate}
\end{proof}

\begin{Lemma}\label{EL3}
    Let $\delta\geq 1$ and suppose that $f\in L^r(\Omega)\setminus\{0\}$ is a non-negative function in $\Omega$ for some $r\geq 1$. For each fixed $n\in\mathbb{N}$, let $v_n$ be unique weak solution of the problem (\ref{M0}). Then the following conclusions hold:
\begin{enumerate}
    \item[$(a)$] If $r>\frac{N}{2}$, then the sequence $\{v_n\}_{n\in\mathbb{N}}$ is uniformly bounded in $L^\infty(\Omega).$
    \item[$(b)$] If $1\leq r<\frac{N}{2}$, then the sequence $\{v_n\}_{n\in\mathbb{N}}$ is uniformly bounded in $L^s(\Omega),$
    where $s=\frac{Nr(\delta+1)}{N-2r}$.
\end{enumerate}
\end{Lemma}
\begin{proof}
\begin{enumerate}
\item[$(a)$] The proof follows from Lemma \ref{EL5}.

\item[$(b)$] We divide the proof into two cases $\delta=1$ and $\delta>1.$\\
\textbf{Case-I:} Suppose $\delta=1$. {If $r=1$, then by taking into account Lemma \ref{emb}, the conclusion follows by choosing test function $v_n$ in the weak formulation of \eqref{M0} and using \eqref{lkernel}, \eqref{nkernel}}. Let us assume that $r>1$. By incorporating the test function $\phi=v_n^{2\gamma-1}$ ($\gamma>1$ to be determined later) in the weak formulation of (\ref{M0}) and utilizing (\ref{lkernel}), (\ref{nkernel}), we get
\begin{align*}
 \alpha\int_\Omega|\nabla v_n^\gamma|^2\, dx+\Lambda^{-1}\underbrace{\int_{\mathbb{R}^N}\int_{\mathbb{R}^N} \frac{(v_n(x)-v_n(y))(\phi(x)-\phi(y))}{|x-y|^{N+2s}}\, dxdy}_{\geq 0}\leq\int_\Omega \frac{\phi T_n(f)}{(v_n+\frac{1}{n})}\, dx,
\end{align*}
which reveals that
\begin{align*}
\int_\Omega |\nabla v_n^\gamma|^2\, dx\leq C\int_\Omega v_n^{2\gamma-2}f\, dx\leq C||f||_{L^r(\Omega)}\left (\int_\Omega v_n^{2r'(\gamma-1)}\, dx\right )^\frac{1}{r'}
\end{align*}
{for some constant $C>0$ independent of $n$.} Applying Lemma \ref{emb} in the above estimate, one has
\begin{align}\label{M3}
    \left (\int_\Omega v_n^{2^*\gamma}\, dx\right )^\frac{2}{2^*}\leq C||f||_{L^r(\Omega)}\left (\int_\Omega v_n^{2r'(\gamma-1)}\, dx\right )^\frac{1}{r'},
\end{align}
for some constant $C>0$ independent of $n$. {Here $2^*=\frac{2N}{N-2}$.}
    We choose $\gamma$ such that $2^*\gamma=2r'(\gamma-1)$, i.e., $\gamma=\frac{r(N-2)}{N-2r}>1$. Using $r<\frac{N}{2}$, the inequality (\ref{M3}) gives that $\{v_n\}_{n\in\mathbb{N}}$ is bounded in $L^s(\Omega)$, where $s=\frac{2Nr}{N-2r}.$\\
\textbf{ Case-II:} We suppose  $\delta>1$. When $r=1$, by choosing $v_n^\delta$ as a test function in the weak formulation of (\ref{M0}) and using the properties (\ref{lkernel}), (\ref{nkernel}), we conclude that $\{v_n^\frac{\delta+1}{2}\}_{n\in\mathbb{N}}$ is uniformly bounded in {$W^{1,2}_0(\Omega)$.} Then by Lemma \ref{emb}, the result follows for $r=1.$ Therefore, let us assume $1<r<\frac{N}{2}$. Treating $v_n^{2\gamma-1}$ ($\gamma>\frac{\delta+1}{2}$ to be determined later) as a test function in the weak formulation of the equation (\ref{M0}), we deduce
\begin{align*}
\int_\Omega |\nabla v_n^\gamma|^2\, dx\leq C\int_\Omega v_n^{2\gamma-\delta-1}f\, dx\leq C||f||_{L^r(\Omega)}\left (\int_\Omega v_n^{r'(2\gamma-\delta-1)}\, dx\right )^\frac{1}{r'},
\end{align*}
for some constant $C>0$ independent of $n$.
Applying Lemma \ref{emb} in the above estimate, we get
\begin{align}\label{M4}
     \left (\int_\Omega v_n^{2^*\gamma}\, dx\right )^\frac{2}{2^*}\leq C||f||_{L^r(\Omega)}\left (\int_\Omega v_n^{r'(2\gamma-\delta-1)}\, dx\right )^\frac{1}{r'},
\end{align}
for some constant $C>0$ independent of $n$ and {$2^*=\frac{2N}{N-2}$.}
 We choose $\gamma$ in such a way that $2^*\gamma=r'(2\gamma-\delta-1)$, i.e., $\gamma=\frac{r(\delta+1)(N-2)}{2(N-2r)}>\frac{\delta+1}{2}$ (since $r>1$). From the fact $r<\frac{N}{2}$ and (\ref{M4}) we conclude that $\{v_n\}_{n\in\mathbb{N}}$ is bounded in $L^s(\Omega)$ with $s=\frac{Nr(\delta+1)}{(N-2r)}.$ This completes the proof.
 \end{enumerate}
\end{proof}

\section*{Acknowledgement} P. Garain thanks IISER Berhampur for the Seed grant: IISERBPR/RD/OO/2024/15, Date: February 08, 2024.

\end{document}